\documentclass[10pt]{article}
\usepackage{graphicx}
\usepackage{epsfig}
\usepackage{amsmath}
\usepackage{amssymb}
\usepackage{amsthm}
\usepackage{amsfonts}
\usepackage{amsgen} 
\usepackage{stmaryrd}
\usepackage{mathrsfs}
\usepackage{psfrag}
\usepackage{color}
\graphicspath{ {./} }

\numberwithin{equation}{section}
\newcommand{\us}{u_{\sigma}}

\newcommand{\weakstar}{\overset{\ast}{\rightharpoonup}}
\newcommand{\ue}{u^{\epsilon}}
\newcommand{\uex}{u(x;x_0,\epsilon)}
\newcommand{\sgn}{{\rm sgn}\,}
\newcommand\restr[2]{\ensuremath{\left.#1\right|_{#2}}}
\newcommand{\Rot}{\textrm{Rot}}
\newcommand{\hone}{\mathcal{H}^{1}}
\newcommand{\R}{\mathbb{R}}

\newcommand{\om}{\Omega}
\newcommand{\Div}{{\rm div}\,}

\newcommand{\cof}{{\rm cof}\,}
\newcommand{\adj}{{\rm adj}\,}

\newcommand{\1}{{\bf 1}}

\newcommand{\eps}{\epsilon}
\newcommand{\sca}{\mathcal{A}}
\newcommand{\dist}{\textrm{dist}\,}
\newcommand{\scl}{\mathcal{L}}

\newcommand{\im}{\mathrm{im}\,}

\newcommand{\sch}{\mathcal{H}}

\def\XXint#1#2#3{{\setbox0=\hbox{$#1{#2#3}{\int}$}
     \vcenter{\hbox{$#2#3$}}\kern-.5\wd0}}

\newtheorem{definition}{Definition}[section]
\newtheorem{lemma}[definition]{Lemma}
\newtheorem{theorem}[definition]{Theorem}
\newtheorem{proposition}[definition]{Proposition}

\newtheorem{remark}[definition]{Remark}

\title{A condition for the H\"{o}lder regularity of local minimizers of a nonlinear elastic energy in two dimensions}

\author{Jonathan J. Bevan \footnote{Department of Mathematics, University of Surrey, Guildford, Surrey,  GU2 7XH, UK.  email: j.bevan@surrey.ac.uk}}

\begin{document}
\maketitle

\begin{abstract}  We prove the local H\"{o}lder continuity of strong local minimizers of the stored energy functional \[E(u)=\int_{\om}\lambda |\nabla u|^{2}+h(\det \nabla u) \,dx\] 
subject to a condition of `positive twist'.  The latter turns out to be equivalent to requiring that $u$ maps circles to suitably star-shaped sets.  The convex function $h(s)$ grows logarithmically as $s\to 0+$, linearly as $s \to +\infty$, and satisfies $h(s)=+\infty$ if $s \leq 0$.  These properties encode a constitutive condition which ensures that material does not interpenetrate during a deformation and is one of the principal obstacles to proving the regularity of local or global minimizers.  The main innovation is to prove that if a local minimizer has positive twist a.e\frenchspacing. on a ball then an Euler-Lagrange type inequality holds and a Caccioppoli inequality can be derived from it.  The claimed H\"{o}lder continuity then follows by adapting some well-known elliptic regularity theory.   We also demonstrate the regularizing effect that the term $\int_{\om} h(\det \nabla u)\,dx$ can have by analysing the regularity of local minimizers in the class of `shear maps'.  In this setting a more easily verifiable condition than that of positive twist is imposed, with the result that local minimizers are H\"{o}lder continuous.\\[1mm] Keywords: Nonlinear elasticity, H\"{o}lder regularity
\\[1mm]
\end{abstract}

\section{Introduction}
\label{intro}
In this paper we consider the question of regularity of local minimizers of functionals representing the stored energy of two-dimensional elastic bodies.  Ball notes in \cite{Ba02} that this is one of a number of outstanding open problems in the field, and while we believe that the results presented here are a positive contribution towards improving the regularity of elastic energy minimizers, we do not claim to be able to prove that the latter are smooth, which in the sense of \cite[Section 2.3]{Ba02} would be the ideal outcome of such an effort.  Instead, our goal is the more modest one of improving the regularity of continuous local minimizers to H\"{o}lder continuity.

Energy functionals in nonlinear elasticity are typically of the form 
\[ E(u)=\int_{\om} W(\nabla u(x))\,dx \]
where $W$ is polyconvex, that is $W(F)$ is a convex function of $(F,\det F)$, and where 
\begin{equation}\label{blowup}W(F,\delta) \to + \infty  \ \textrm{ as } \delta \to 0+\end{equation}
for each fixed $2 \times 2$ matrix $F$ with $\det F > 0$.    The constitutive condition \eqref{blowup} models the physical reality that compressing material to zero volume ought to incur an infinite energetic cost.  That such a condition could be captured and embedded in an existence theory using polyconvex functions was first realized by Ball \cite{Ba77}, whose well-known work has since given rise to a rich literature on the topic.   Here, $\om$ is a bounded domain in $\R^2$ with Lipschitz boundary representing the region occupied by the elastic material in a reference configuration.

In circumstances where \eqref{blowup} does not hold the regularity or partial regularity of minimizers of polyconvex functionals has been studied by several authors, including but not limited to \cite{GM,FH,EM,CY,H,Gu}.   When \eqref{blowup} is imposed, or if the intention is to somehow faithfully approximate it, fewer results are available.  These can be divided according to whether full or partial regularity is proven.  The  work \cite{FS} is of the former type and focuses on planar stored-energy functions of the form
\begin{equation}\label{fs1}W(F)=g(|F|^2)+h(\det F)  \ \ \ \ \ \det F \geq 0, \end{equation}
where $g,h: [0,+\infty) \to [0,+\infty)$ are of class $C^1$.  Crucially, $g$ and $h$ are finite, meaning in particular that \eqref{blowup} cannot hold;  rather, the extended real valued problem is approximated by not specifying $h(0)$, which could therefore be as large as desired.  Under suitable growth and other assumptions, including that $u \in W^{1,2}(\om,\R^2)$ has finite energy $E(u)$ and the so-called energy-momentum equations 
\[ \Div \left((\nabla u)^{T} D_{F}W(\nabla u) - W(\nabla u) \1\right)=0 \ \ \textrm{in } \mathcal{D}'(\om)\]
hold, $u$ is shown to be H\"{o}lder continuous, with H\"{o}lder exponent depending on the parameters appearing in the growth condition.
Fuchs and Seregin point out in \cite[Remark 2]{FS} that the same result can be obtained when $h(0)=+\infty$ provided the assumptions on $u$ are extended to include $\det \nabla u > 0$ a.e\frenchspacing. in $\om$ and 
\[ \int_{\om} (\det \nabla u)^{-(1+\eps)} \,dx < +\infty\]
for some $\eps >0$.    Unfortunately, the stored-energy functionals we consider in this paper violate the growth conditions given in \cite{FS}, so that even if the condition $(\det \nabla u)^{-1} \in L^{1+\eps}(\om)$ were to hold---and there seems to be no a priori way to tell whether it does or not---the proof given in \cite{FS} does not apply.

In \cite{FS0} and \cite{FR}, the partial regularity of minimizers of functionals with integrands typified by \eqref{fs1} is proven.  Moreover, a sequence of such minimizers is used to strongly approximate in a suitable Sobolev norm a minimizer, $u$ say, of a functional whose stored-energy satisfies \eqref{blowup}.
The process does not confer partial regularity on the limiting minimizer $u$.

More recently, Foss \cite{Fo} uses a blow-up technique to establish a partial regularity result for minimizers $u$, say, of functionals in which \eqref{blowup} is operational.   For his result to hold $u$ is required to satisfy an equiintegrability condition, referred to as (REP), which is phrased in terms of an excess quantity.   Working in $W^{1,p}(\om,\R^2)$, with $p>8$ and $\om \subset \R^2$, and by slightly corrupting Foss's notation, define for $x_0 \in \om$, $R>0$ and $M \subset B(x_0, R/2)$, 
the `excess'
\begin{align*} U(x_0,R;M):= \frac{1}{|B(x_0,R)|} \int_{M} \Bigg\{ & |\nabla u - (\nabla u)_{x_0,R}|^2 + |\nabla u - (\nabla u)_{x_0,R}|^p + \\ & +  \left|\frac{1}{\det\nabla u} - \frac{1}{(\det \nabla u)_{x_0, R}}\right|^{2} \Bigg\}\,dx.
\end{align*}
The term involving the determinant reflects the choice made for the function $h$ in \cite{Fo}, namely $h(s)=s^{-2}$ for $s>0$.   Property (REP) is then that for some $L^{\infty}$ function $Q$ and for each $\beta >0$ the inequality
\[ U(x_0,R;M) \leq \beta U(x_0,R;B(x_0,R)) Q((\nabla u)_{x_0, R})\]
holds whenever both $U(x_0,R;B(x_0,R))$ and $|M|/|B(x_0,R)|$ are sufficiently small.  With this in force, $u$ is shown to be $C^{1,\alpha}$ on an open set of full measure in $\om$.   It is not known whether property (REP) actually holds for minimizers of stored-energy functionals.   We note that the minimizer in \cite{Fo} is automatically H\"{o}lder continuous on $\om$ by Sobolev embedding in conjunction with the assumption $p > 8$.

The proposal we make in this paper is to prove local H\"{o}lder regularity by replacing the technical condition (REP) with another, simpler condition which admits a straightforward geometric interpretation.   The aim is distinct from the works on partial regularity cited above in the sense that we prove a lesser degree of regularity but on a larger set, and in a situation where $\eqref{blowup}$ is in operation.   To describe the condition we impose we first fix the class of stored-energy functions to which these new arguments apply.   For $\lambda > 0$ let
\[ W(F) = \lambda  |F|^2 + h(\det F)\]
 defined on $2 \times 2$ matrices $F$, where, for fixed constants $0< c_1 < 1 < c_2 < +\infty$ and $l, m > 0$,  $h:\R \to [0,+\infty]$ is given by
\begin{displaymath} h(s) = \left\{ \begin{array}{l l}+ \infty & \textrm{ if } s \leq 0  \\
|\ln s| & \textrm{ if } s \in (0,c_1)  \\ 
\theta(s) & \textrm{ if } s \in [c_1,c_2] \\
l s + m & \textrm{ if } s \in (c_2, +\infty).\end{array}\right.
\end{displaymath}
The function $\theta$ is chosen so that $h$ is convex and of class $C^1$.  We consider maps such that $E(u) = \int_\om W(\nabla u) \,dx$ is finite, which, in view of the definition of $h$, immediately implies that $\det \nabla u > 0$ a.e\frenchspacing. in $\om$.  By \cite{VG77}, or by \v{S}ver\'{a}k's well-known regularity result \cite[Theorem 5]{Sv88}, such $u$ are in particular continuous.   Thus in our case the improvement in regularity, when it occurs, is from continuous to H\"{o}lder continuous.   The condition we use to ensure the improvement is based on the nonnegativity of 
\[ t(x,x_0,u):= \adj \nabla u(x) (u(x)-u(x_0)) \cdot \frac{x-x_0}{|x-x_0|}, \] 
a quantity which for each $u \in W^{1,2}(\om, \R^2) \cap C^{0}(\om, \R^2)$ is defined a.e\frenchspacing. on a subset of $\om \times \om$.  We shall later refer to $t(x,x_0,u)$ as the twist of $u$ at $x$ relative to $x_0$.  

For smooth diffeomorphisms $v: \R^2 \to \R^2$, say, with $\det \nabla v > 0$ it is easy to show that 
\[ t(x,x_0,v) = \det \nabla v (x_0) |x-x_0| + o(|x-x_0|) \quad \textrm{ as } |x-x_0| \to 0.\]
In particular, $t(x,x_0,v) > 0$ for all $x$ sufficiently close to $x_0$.  For a general $u$ belonging to $W^{1,2}(\om, \R^2) \cap C^{0}(\om, \R^2)$ the same need not be true and must instead be hypothesized.   For a given $u$ we suppose that there is $z \in \om$ and $\delta > 0$, $r'>0$ such that
\begin{equation}\label{posball}
t(x,x_0,u) \geq 0 \quad \textrm{ for a.e. } x \in B(x_0,r') \textrm{ and a.e. } x_0 \in B(z,\delta).
\end{equation}
Suppose we fix $x_0$ in $B(z,\delta)$ for which $t(x,x_0,u) \geq 0$ for a.e\frenchspacing. $x$ in $B(x_0,r')$.    It is shown in Section \ref{starshape} that this condition is equivalent to requiring that $u$ be locally star-shaped in the sense that for a.e\frenchspacing. $R \in (0,r')$ the image $u(S(x_0,R))$ of a circle $S(x_0,R)$ centred at $x_0$ and of radius $R$ is star-shaped: see Definition \ref{defstarshape}.     To do so we rely on the powerful technical machinery of M\"{u}ller and Spector \cite{MS95}.  The key point is that the interaction of $u(S(x_0 ,R))$ with rays emanating from $u(x_0)$ can be understood using properties of the degree.  Further details can be found in Section \ref{starshape}.

When condition \eqref{posball} holds we show in Section \ref{prm} that local minimizers of the functional $E$ are H\"{o}lder continuous on compact subsets of $B(z,\delta)$.   Here, $u$ is a (strong) local minimizer if $E(v) \geq E(u)$ for all $v$ such that $E(v)<+\infty$ and $||v-u||_{\infty;\om}$ is sufficiently small.   The argument uses $t(x,x_0,u) \geq 0$ in two essential ways, by 
\begin{itemize}
\item[(i)] ensuring that a suitable class of outer variations $\ue = u + \eps \varphi$ satisfies $E(\ue) < + \infty$ for all sufficiently small and negative $\eps$, and
\item[(ii)] establishing that the limit 
\[\limsup_{\eps \nearrow 0} \frac{E(\ue)-E(u)}{\eps} \leq 0\]
holds and results in a variational inequality to which elliptic regularity theory can be applied.
\end{itemize}
We remark that the logarithmic growth of $h(\det F)$ as $\det F \to 0+$ plays a pivotal role in (ii); however, it may in future be possible to generalize the technique to other types of singularity.   Indeed, in Section \ref{shear} we do this, albeit in a restricted class of competing functions.

Section \ref{shear} focuses on an application of some of the ideas of Section \ref{prm} to a particular class of maps---the so-called shear maps.  These allow us to treat a much wider class of stored-energy functions in the sense that the $h(\det \nabla u)$ term of the stored-energy function $W$ is required to obey much weaker growth conditions:  see hypotheses (H0)-(H4) of Section \ref{shear} for details.    The term $\int h(\det \nabla u)$ has a regularizing effect on local minimizers, or at least its presence is not an impediment to the improvement of regularity.   The gain in regularity is  tied to imposing a condition (see \eqref{lipshear}) which allows us to construct outer variations as before, and subsequently to adapting and applying a version of elliptic regularity theory.      

The organisation of the paper is as follows.  After introducing notation in Section \ref{notation}, we give in Section \ref{starshape} the geometric characterisation of the condition $t(x,x_0,u) \geq 0$.    
Section \ref{prm} is broken into three subsections, the first and shortest of which introduces the main functionals to be studied and ends with a brief proof of the existence of at least one energy minimizing deformation.   The rest of Section \ref{prm} consists in establishing a variational inequality (Section \ref{elineq}) to which an adapted elliptic regularity theory applies (Section \ref{holreg}).  The paper concludes in Section \ref{shear} with an analysis of the regularity of local minimizers in the class of shear maps, followed by a short appendix containing two ancillary results.

\subsection{Notation}\label{notation}
The standard notation $B(x,r)$ for a ball of radius $r$ and centre $x$ will be used, and its boundary $\partial B(x,r)$ will be written $S(x,r)$. The terms null set and a.e. will refer to $\scl^{2}$ measure;  in all other cases the relevant measure, most often $\sch^{1}$, will be explicitly referred to through `$\sch^{1}-$null' and $\sch^{1}-$a.e. respectively.  $\textrm{Rot}(\psi)$ will denote the matrix 
representing rotation anticlockwise through $\psi$ radians, with the particular letter $J$ reserved for $\textrm{Rot}(\pi/2)$.   The unit vector $e(\theta)$ is defined to be $e(\theta):=(\cos \theta, \sin \theta)^{T}$.  Our notation for norms in Lebesgue and Sobolev spaces, as well as for the spaces themselves, is conventional.  For clarity, we write $||u||_{p;\om}$ for the $L^{p}$ norm $(\int_{\om} |u|^p \,dx)^{1/p}$ and $||u||_{1,p;\om}$ for the Sobolev norm
 $(\int_{\om} |u|^p + |\nabla u|^p \,dx)^{1/p}$ when $1\leq p < +\infty$, with the usual adjustments for the case $p=+\infty$.   Here, $\om$ is a domain in $\R^2$ and $\nabla u$ is the weak derivative of $u$.   In the case of planar maps $u: \om \to \R^2$ we write 
\[ \nabla u = \left(\begin{array}{ c c} u_{\scriptscriptstyle{1,1}} & u_{\scriptscriptstyle{1,2}} \\ u_{\scriptscriptstyle{2,1}} & u_{\scriptscriptstyle{2,2}} \end{array}\right)\]
where $u(x)=(u_{\scriptscriptstyle{1}}(x),u_{\scriptscriptstyle{2}}(x))^{T}$ and $u_{\scriptscriptstyle{i,j}}:=\partial u_{i} / \partial x_{j}$.
Other notation is introduced as and when it is needed.


\section{Positive twist, condition (INV) and local star-shapedness}\label{starshape}

In this section we frame the conditions needed to derive the main regularity results of the paper.  
Let $u_0$ be a homeomorphism of $\om$ onto $u_0 (\om)$ and assume that there is at least one mapping $u$ in the class
\begin{equation}\label{a1} \sca_1:=\left\{u \in W^{1,2}(\om,\R^2): \ \ \det \nabla u > 0 \ \textrm{a.e.}, \ \ \restr{u}{\partial \om} = \restr{u_0}{\partial \om}\right\}.
\end{equation}
Maps belonging to $\sca_1$ will be referred to as admissible maps.   To each such map and each point $x_{0} \in \om$ we associate a map $x \mapsto t(x,x_0,u)$, which we call the relative twist of $u$ about $x_0$, and which is defined in terms of a general element of $W^{1,1}(\om,\R^2)$ as follows:

\begin{definition}\label{deftwist} Let $u \in W^{1,1}(\om,\R^{2})$ and let $x_{0} \in \om$.  Define the twist $t(x,x_0,u)$ of $u$ at $x$ relative to $x_{0}$ by 
\begin{equation}\label{twist} t(x,x_0,u):= \adj \nabla u(x) (u(x)-u(x_{0})) \cdot \frac{x-x_{0}}{|x-x_{0}|}.\end{equation}
\end{definition}

Notice that $t(x,x_0,\om)$ is only defined for a.e. $x$ in $\om$ (since the same is true of $\nabla u$).  However, by convention, we shall refer to $t(x,x_0,u)$ as a function rather than an equivalence class.  Also, although the twist is defined relative to a point $x_0$ in $\om$ and so, strictly speaking, should be referred to as a relative twist, we shall nevertheless refer to it simply as twist, hoping that no confusion arises.

We postpone until Section \ref{prm} both the motivation for studying the relative twist of an admissible map and an explanation for requiring that it be nonnegative on an appropriate subset of $\om \times \om$.   Instead, we focus here on the geometric consequences of requiring that the twist be locally nonnegative.  To be more specific,  it turns out that requiring $t(x,x_0,u) \geq 0$ for a.e\frenchspacing. $x \in B(x_0, r')$ and for some $x_0 \in \om$, where $r' < \dist(x_0,\partial \om)$, is equivalent to the statement that $u$ maps circles $S(x_0, R)$ to star-shaped sets for a.e\frenchspacing. $R \in (0,r')$.  To prove it we exploit the fact that admissible maps automatically have a continuous representative (which we henceforth identify with the map itself), which in turn allows us to apply some of the machinery of \cite{MS95}.    The following technical result records some of the properties of maps in $\sca_1$.   

\begin{proposition}\label{techprop1} Let $u \in \sca_1$ as defined in \eqref{a1}.  Then 
\begin{itemize}\item[(i)] $u$ can be identified with its continuous representative;
\item[(ii)]$u$ has the $N-$property;
\item[(iii)] $u$ is $1-1$ $\scl^{2}-$almost everywhere;
\item[(iv)] for each $x_{0} \in \om$ there is a set $N_{0} \subset \R^+$ with $\scl^1(N_0)=0$ such that
\[\restr{u}{S(x_{0},R)} \ \textrm{is} \  1-1 \ \hone-\textrm{a.e. for all} \ R \in (0,R_{0}) \setminus N_0,\] 
where $R_{0} := \dist(x_{0},\partial \om)$, and
\item[(v)] $u$ satisfies condition (INV).
\end{itemize}
\end{proposition}
\begin{proof} (i):  This is \cite[Theorem 4]{Sv88}, or \cite[Theorem 2.32]{VG77}.  (ii):  See \cite[Theorem 5.32]{FG95}, or \cite[Theorem 6]{Sv88}.  It helps to read the latter in conjunction with the proof of \cite[Lemma 5 (i)]{Sv88}.   \newline (iii):  The hypothesis $\restr{u}{\partial \om} = \restr{u_0}{\partial \om}$, where $u_0$ is a homeomorphism, means that \cite[Lemma 5 (i)]{Sv88} applies directly.  Alternatively, apply (v) and \cite[Lemma 3.4]{MS95} to reach the same conclusion.  We also present a simple, direct proof, as follows.   Firstly, by \cite[Theorem 5.21]{FG95}, the set
\begin{equation}\label{om1} \om_1:=\{x \in \om: \ u \textrm{ has a classical derivative at } x \}\end{equation}
is a set of full measure in $\om$.    By excluding a further subset of measure zero on which $\det\nabla u =0$, and then if necessary relabelling $\om_1$, we can assume that $\det \nabla u(x) > 0$ for all $x \in \om_1$.   Part (i) and \cite[Lemma 5.9]{FG95} then imply that for each $x$ belonging to $\om_1$ and for all sufficiently small $R>0$, $u(x+h) \neq u(x)$ for all $h \in \overline{B(0,R)}$ and that $d(u,B(x,R),u(x))= \sgn \det \nabla u (x) = 1$.  The former ensures that $u(x) \notin u(S(x,R))$ and hence that the degree as written is well-defined. 
Let $y \in u_0(\om) \setminus u_{0}(\partial \om)$.   Since $u$ agrees with the homeomorphism $u_0$ on $\partial \om$ then standard properties of the degree mean that $d(u,\om,y)=d(u_0,\om,y)=1$, and hence that $u^{-1}(y)$ is nonempty.  Suppose for a contradiction that $u^{-1}(y) \cap \om_{1}$ contains at least two points $x_1 \neq x_2$.   By the above, there are balls $B_{1}:=B(x_{1},R)$ and $B_{2}=B(x_{2},R)$ such that $d(u,B_1,y)=d(u,B_2,y)=1$.   By the excision and domain decomposition properties of the degree we must therefore have
\[ d(u,B,y)=d(u,B_1,y)+d(u,B_2,y) + d(u,D,y),\]
where $D:=\om \setminus \overline{B_1 \cup B_2}$.   But the left-hand side of this equation is $1$, while the right-hand side is, in view of $\det \nabla u > 0$ a.e., Heinz's formula (see e.g. \cite[Proposition 1.7]{FG95}), and the values of $d(u,B_i,y)$ aready indicated for $i=1,2$, at least $2$, which is a contradiction.   Hence $u^{-1}(y) \cap \om_1$ contains one point or is empty.  Therefore by restricting $u$ to $\om_1$ we see that $u$ is $1-1$ a.e..
\newline (iv) This follows from \cite[Proposition 2.8 (iii)]{MS95} and from (iii) above:  see \cite[Lemma 3.1, Step 1, equation (3.9)]{MS95}.
\newline (v) Condition (INV) is given by \cite[Definition 3.2]{MS95}.  To verify it we must show that for each $x_{0} \in \om$ there exists an $\scl^1$ null set $N_{0}$ such that,  for all $r \in (0,R_0) \setminus N_0$, $\restr{u}{S(x_0, R)}$ is continuous, 
\begin{itemize}\item[(I)] $u(x) \in \im_{T}\left(u,B(x_{0},R)\right) \cup u(S(x_0, R))$ for $\scl^2-$a.e. $x \in \overline{B(x_0, R)}$, and 
\item[(II)]  $u(x) \in \R^2 \setminus  \im_{T}\left(u,B(x_{0},R)\right)$ for $\scl^2-$a.e. $x \in \om \setminus \overline{B(x_0, R)}$.
\end{itemize}
By (i), the continuity of $u$ on $S(x_{0},R)$ for all $R \in (0,R_0)$ is assured, so the first part is automatically true (with $N_0$ empty).  Condition (I) holds by appealing to \cite[Theorem 3, Cor. 1 (ii)]{Sv88}, where, in their notation (and in view of the constraint $\det \nabla u > 0$ a.e.) the set $E(u,B(x_0, R))$ replaces $\im_{T}(u,B(x_0, R))$.  To see that condition (II) holds it is sufficient to show that the set $\om_2:=\{x \in \om \setminus \overline{B(x_0, R)}: \ u(x) \in  \im_{T}(u,B(x_0, R))\}$ has $\scl^2$ measure zero.   Let $x \in \om_{2}$.  Then $x \notin S:=S(x_0, R)$ and so, if $y:=u(x)$, the degree $d(y):=\deg(u,S,y)$ is well defined and, by hypothesis, satisfies $d(y)=1$.   By properties of the degree there is at least one $x' \in B(x_{0},R)$ such that $u(x')=y$, where $x' \neq x$ in particular.   Since $u$ is $1-1$ a.e by (iii), and in the notation of that part of the proof, we must have $x \in \om \setminus \om_{1}$, which, since the latter set is $\scl^2$ null, concludes the proof of (v). \qed   
\end{proof}

We now turn to the derivation of necessary and sufficient conditions for the local nonnegativity of the twist $t(x,x_0,u)$ which apply in our setting.   To begin with we note that, for sufficiently regular maps $u$, the condition $\det \nabla u > 0$ implies that $t$ is positive everywhere in a sufficiently small ball about $x_{0}$.   For later use we record the following result, whose proof is straightforward and is therefore omitted.

\begin{proposition}\label{diffeo} Let $x_0 \in \om$ and let $u: \om \to \R^2$ be a diffeomorphism in a neighbourhood of $x_{0}$. Then
\[ t(x,x_0,u) = \det \nabla u(x_{0}) |x-x_{0}| + o(|x-x_{0}|) \ \ \textrm{as} \ |x-x_{0}| \to 0.\]
\end{proposition}

It follows that if $\det \nabla u(x_{0})> 0$ then the twist of $u$ relative to $x_{0}$ is necessarily positive on a sufficiently small ball around $x_{0}$.  The challenge is to extend this result to maps $u$ belonging, for example, to $\sca_1$, which clearly need not be $C^{1}$ and where $\det \nabla u (x_{0})$ need not be defined pointwise.   In such cases it is possible to construct maps for which $t(x,x_0,u) < 0$ for $x$ belonging to a set of positive measure.   Rather than give these examples we prefer to avoid them altogether by appealing to Lemma \ref{starshape1} below, which is phrased in terms of local star-shapedness:

\begin{definition}\label{defstarshape} Let $u: \om \to \R^{2}$ be continuous and let $S(x_{0},R) \subset \om$.   Then $u(S(x_{0},R))$ is star-shaped with respect to $u(x_{0})$ if:
\begin{itemize}\item[(i)] $a:=u(x_{0})$ belongs to $\R^2 \setminus u(S(x_{0},R))$, and
\item[(ii)] each half-line 
\[[a,u(x)]_{+}:=\{ a + \mu(u(x)-a): \ \ \mu \geq 0\}\]
is such that $u(S(x_{0},R)) \cap [a,u(x)]_{+}$ forms a connected subset of $[a,u(x)]_{+}$ for $\sch^{1}-$a.e. $x \in S(x_{0},R)$.
\end{itemize}
\end{definition}

\begin{lemma}\label{starshape1}
Let $t(x,x_0,u)$ be as per Definition \ref{deftwist} and let $u \in \sca_1$, as defined in 
\eqref{a1}.   Let $x_0 \in \om_1$ as defined in \eqref{om1}, so that $u$ is classically differentiable at $x_0$, and suppose $\delta_0 < \dist(x_0,\partial \om)$.  Then 
\begin{itemize}\item[(A)]$t(x;u,x_{0}) \geq 0$ for a.e. $x \in B(x_0,\delta_0)$ if and only if
\item[(B)] $u(S(x_{0},R))$  is star-shaped with respect to $u(x_{0})$ for almost every $R \in (0,\delta_0)$.
\end{itemize}
\end{lemma}
\begin{proof}  Let $a=u(x_{0})$ and, for $x \neq x_{0}$,  $\nu(x):=(x-x_{0})/|x-x_{0}|$.    Since $x_{0}$ and $u$ are fixed we abbreviate $t(x,x_0,u)$ to $t(x)$.  Let $R \in (0,\delta_0)$ and set $S:=S(x_0,R)$. We further assume that $\nabla u$ coincides almost everywhere with its approximate derivative $\textrm{ap}\,Du$ (see \cite[Section 6.1.3]{EG92}, for example), and by a slight abuse of notation continue to denote the latter by $\nabla u$.    Thus, in the notation of \cite[Eq. (3.14)]{MS95}, the induced normal on $u(S)$ is given by
\begin{equation}\label{indnorm} \tilde{\nu}(u(x))= \frac{\cof \nabla u(x) \nu(x)}{|\cof \nabla u(x) \nu(x)|}, \quad \quad \quad x \in S.\end{equation}
Note that since $\det \nabla u > 0$ a.e., $|\cof \nabla u(x) \nu(x)|\neq 0$ a.e. in $\om$.   In terms of $t$, we have
\begin{eqnarray*}  t(x) & = & \adj \nabla u(x) (u(x)-a) \cdot \nu(x) \\  & = & (u(x) - a) \cdot \cof \nabla u(x) \nu(x),
\end{eqnarray*}
so that 
\begin{equation}\label{t:geom1}t(x) \geq 0 \ \ a.e. \ \  \textrm{if and only if} \ \  (u(x)-a)\cdot \tilde{\nu}(u(x)) \geq 0\ \ a.e.\end{equation}  
Roughly speaking, the rightmost inequality in \eqref{t:geom1} says that the curve $u(S)$ turns monotonically while $\nu(x)$ traverses a circle. 

  By Proposition \ref{techprop1} and \cite[Lemma 3.5 (ii)]{MS95}, we may assume that the degree $d(y):=\deg(u,S,y)$ satisfies $d(y) \in \{0,1\}$ for all $y \in \R^{2} \setminus u(S)$.  Thus, in the notation of \cite[Lemma 3.5]{MS95}, the set 
$U_{1}:=\{y \in R^2 \setminus u(S): \ d(y) \geq 1\}$ coincides with the topological image $\textrm{im}_{T}\,(u,B(x_{0},R))$.   In particular, by \cite[Step 6, Lemma 3.5]{MS95}, $\tilde{\nu}(u(x))$ coincides $\sch^{1}-$a.e. with the generalized exterior normal on the set $\partial^\ast U_1=u(S)$.  Note that the latter holds up to a set of $\sch^{1}-$measure zero:  this follows from \cite[Step 3, Lemma 3.5]{MS95}, where it is shown that $U_{1}$ is a set of finite perimeter (which itself follows from the fact that the degree $d$ is a BV function), so that the reduced boundary $\partial^\ast U_{1}$ differs from $\partial U_{1}$ only by an $\sch^{1}-$null set.  To conclude the preliminaries we relate the generalized exterior normal to the tangent to $u(S)$ as follows.   Firstly, by writing
\[ \nabla u(x) = u_{R}(x) \otimes  \nu(x) + u_{\tau}(x) \otimes J\nu(x)\]
for $x \neq x_{0}$, we find that, by a slight abuse of notation 
\[ t(x) = J (u(x)-a) \cdot u_{\tau}(x)\]
where $u_{\tau}(x)=\frac{1}{R}\partial_{\theta} u(x_{0}+R e(\theta))$, $x=x_{0}+R e(\theta)$, and $e(\theta) = (\cos \theta, \sin \theta)^{T}$.  Thus $t$ has a representation in local polar coordinates.   Further, it is clear from $\cof \nabla u(x) \nu(x) = J^{T}u_{\tau}(x)$ and \eqref{indnorm} that
\begin{equation}\label{j:tildenu}\tilde{\nu}(x) = \frac{J^{T}u_{\tau}(x)}{|u_{\tau}|} \quad \quad \sch^{1}-a.e. \ x \in S.\end{equation}
Thus the generalized exterior normal to $U_{1}=u(B(x_{0},R))$ at $u(x)$ is obtained by scaling and then rotating $u_{\tau}(x)$ clockwise through $\pi/2$ radians at $\sch^{1}-$a.e. $x$ in $S$.   

\vspace{2mm}
\noindent($\mathbf{t \geq 0 \Rightarrow u(S)}$\textbf{ is star-shaped}).    

\vspace{1mm}
\noindent To see that part (i) of Definition \ref{defstarshape} holds  we recall that for each $x_0 \in \om_1$ there is $\delta_1 > 0$ such that 
\[u(x_0+h) \neq u(x_0) \ \ \ \forall h \in \overline{B(0,\delta_1)}.\]
By continuity, it follows that $u(x_0) \notin u(S(x_0,R))$ for all $R \in (0,\delta_0)$.   Therefore part (i) of Definition \ref{defstarshape} holds.

\vspace{1mm}
\noindent To prove that part (ii) of Definition \ref{defstarshape} holds we make use of the assumption $t \geq 0$ in conjunction with some topological observations.  To begin with, by  replacing $u(x)$ with $u(x)-a$, we may assume that $a=0$, and hence that $t(x)=u(x) \cdot Ju_{\tau}(x)$.   Define the functions $\rho(\theta)$ and $\sigma(\theta)$ by 
\begin{eqnarray} \nonumber \rho(\theta) & := & |u(x_0+Re(\theta))| \\
\label{j:sigmadef} e(\sigma(\theta)) & := & \frac{u(x_0+ Re(\theta))}{|u(x_0+ Re(\theta))|},
\end{eqnarray}
so that
\[u(Re(\theta)) = \rho(\theta)e(\sigma(\theta)).\]
Note that $\sigma$ is so far only defined up to a multiple of $2\pi$.  To fix one particular $\sigma$ we argue as follows.  Firstly, since part (i) of Definition \ref{defstarshape} holds, we can suppose that $\rho(\theta)\geq c$ for some constant $c>0$ and for all $\theta \in [0,2 \pi)$.   Further, since $\theta \mapsto u(Re(\theta))$ belongs to $W^{1,2}([0,2 \pi),\R^2)$, it follows that $\rho$   also belongs to that class.  Next, note that \eqref{j:sigmadef} together with the continuity of both $u$ and $\rho$ implies that $\sigma$ is locally continuous, for example by viewing it as 
\[ \sigma(\theta) = \cos^{-1} \big(u_{1}(Re(\theta))/ \rho(\theta)\big)\]
suitably interpreted.  It follows that if we fix $\sigma$ in a neighbourhood of some $\theta=0$ and extend this representative to $[0, 2\pi)$ then the resulting function, again denoted by $\sigma$, is uniquely defined.  It is now easy to see that $\sigma$ belongs to $W^{1,2}([0,2\pi),\R)$, and that 
\begin{equation}\label{j:sigmadot} t(Re(\theta)) = \frac{\rho^2 (\theta)\dot{\sigma}(\theta)}{R} \ \ \ a.e. \ \theta \in [0,2\pi). \end{equation} 
In particular, since $t \geq 0$ a.e. by assumption, it follows that $\dot{\sigma} \geq 0$ a.e.

Now suppose for a contradiction that part (ii) of  Definition \eqref{defstarshape} does not hold.  
Then there is $x \in S$ such that the half-line $[0,u(x)]^{+}$ disconnects $u(S) \cap [0,u(x)]^+$.  
We claim that this implies that $\sigma$ strictly decreases somewhere.  The proof is broken into three steps, where we make use of the more concise notation \[L_{u(x)}:=[0,u(x)]^+.\]

\vspace{1mm}
\noindent{\textbf{Step 1}}  Since $u(S)$ is compact and does not contain $0$, we can without loss of generality assume $x \in S$ is such that $0 <|u(x)|= \min\{|z|:  \ z  \in u(S) \cap L_{u(x)}\}$.   Let $\Gamma_x$ be the connected component of $u(S) \cap L_{u(x)}$ containing $u(x)$.   The supposition above means that there is at least one other component $\Gamma_y$, say, of $u(S) \cap L_{u(x)}$ containing $u(y)$ such that $\Gamma_x \cap \Gamma_y= \emptyset$
and $\dist(\Gamma_x, \Gamma_y)$ is minimal.   For definiteness, and by changing $y$ if necessary, we can assume $u(y)$ is the closest point in $\Gamma_y$ to $\Gamma_x$.

There are two possibilities for the behaviour of the curve $u(S)$ in a neighbourhood of $u(x)$:  either there is a branch of $u(S)$ containing $u(x)$ and lying in the open sector
\[ C_{u(x),\eps_0 }:=\bigcup_{0 < \psi < \eps_0} L_{\Rot(\psi)u(x)}  \setminus \{0\}\] 
which borders and, for small $\eps_0 > 0$, lies anticlockwise relative to $L_{u(x)}$, or there is a branch with the same properties lying instead in the open sector 
\[C_{u(x),-\eps_0} :=  \bigcup_{0 > \psi > -\eps_0} L_{\Rot(\psi)u(x)}  \setminus \{0\}\]
lying clockwise relative to $L_{u(x)}$.   The second of these is eliminated by the assumption $t \geq 0$ using the same argument as is given at the beginning of Step 2 below.   Therefore we work with the first possibility now and make a remark in Step 2 about the impossibility of the second.

By rotating $L_{u(x)}$ anticlockwise we generate two continuous paths in 
$u(S)$ starting at $u(x)$ and $u(y)$, denoted by $P(x)$ and $P(y)$   respectively, whose construction is given below.   Using the Jordan separation theorem, we write the complement of $u(S)$ in $\R^2$ as a union of connected, open sets, $G_{i}$, $i=0,1,2,\ldots$, which we refer to as components.   Since $u(S)$ is compact there is just one unbounded component, $G_{0}$, say.  Let $G_{1}$  be the (bounded) component containing $0$.

To define $P(x)$ we proceed as follows.  For $\eps_0 >0$ sufficiently small, there is a unique branch $P_0$, say, of $u(S)$ containing $u(x)$, lying strictly in the open sector $C_{u(x),\eps_0}$ and which, in addition satisfies, $\hone(\partial G_{1} \cap P_0) >0$.  Necessarily, points in $P_0$ are in $1-1$ correspondence with angles of rotation $\psi$ in the sense that each set $P_0 \cap L_{\Rot(\psi)u(x)}$ is a singleton for $0 \leq \psi < \eps_0$.      

Let $\psi_0 > \eps_0$ be the first $\psi$ for which 
$P_0 \cap L_{\Rot(\psi)u(x)}$ meets a subset $\Gamma_{x_1}$ of $u(S)$ such that $\hone\big(\Gamma_{x_1} \cap L_{\Rot(\psi_0)u(x)}\big) >0$. By continuity, the point 
\[u(x_1):=\lim_{\psi \nearrow \psi_0} P_0 \cap L_{\Rot(\psi)u(x)}\]
is well defined and, moreover, $u(x_1) \in \Gamma$.   

\vspace{1mm}
\noindent{\textbf{Step 2}}
Let $Q$ be a branch of $u(S)$ with initial point $u(w) \in \Gamma$, $u(w) \neq u(x_1)$, and assume for a contradiction that $Q$ lies in the open sector $C_{u(x_1),-\eps_1}$ for all sufficiently small $\eps_1>0$.  Recall that $C_{u(x_1),-\eps_1}$ borders and lies clockwise relative to $L_{u(x_1)}$.  If there is just one branch $Q$ as described then both $Q$ and $P_0$ have a non-trivial intersection with the boundary of $G_{1}$.   The relation \eqref{j:tildenu} then forces an orientation on both $Q$ and $P_0$, as shown in Figure \ref{z1a}.

\begin{figure}[ht]
\psfragscanon
\psfrag{a}{$L_{u(x)}$}
\psfrag{b}{$L_{u(x_1)}$}
\psfrag{c}{$u(y)$}
\psfrag{d}{$\Gamma_x$}
\psfrag{e}{$u(x)$}
\psfrag{f}[r]{$u(x_1)$}
\psfrag{g}[r]{$u(w)$}
\psfrag{h}{$G_1 \arrowvert 1$}
\psfrag{i}{$\nu$}
\psfrag{j}{$\tau$}
\psfrag{k}{$\nu$}
\psfrag{l}{$\tau$}
\psfrag{Q}{$Q$}
\psfrag{0}{$0$}
\includegraphics[width=0.6\textwidth]{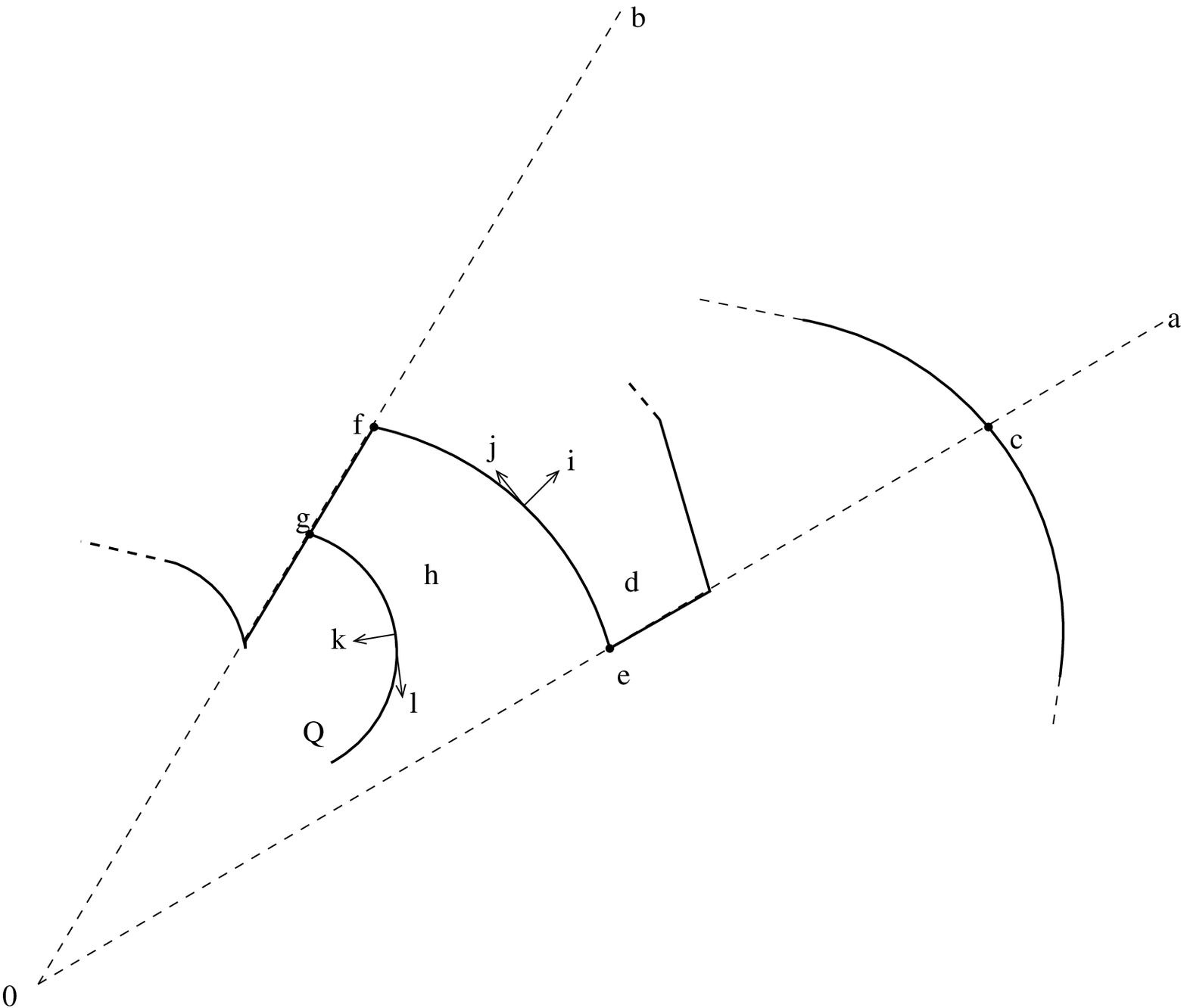}
\caption{The notation $G \, | \, n$ is used to indicate that the degree $d$ is $n$ on the component $G$.  If just one branch $Q$ of $u(S)$ meets $\Gamma_{x_1}$ at $u(w)$ as shown then $\sigma$ decreases along $Q$.}
\label{z1a}
\end{figure}

In view of its inclusion in $C_{u(x_1),-\eps_1}$, it is clear that $\sigma$ strictly decreases along $Q$,  contradicting $\dot{\sigma} \geq 0$.    If there are two or more branches then we can find a component $G'$ of $\R^2 \setminus u(S)$ whose boundary $\partial G'$ has a nontrivial intersection with $Q_1$ and $Q_2$, say.   If $\restr{d}{G'}=0$ then the orientation implied by \eqref{j:tildenu} is shown in Figure \ref{Z2a}
; if $\restr{d}{G'}=1$ then the arrangement is as per Figure \ref{Z2b} 
.   In the first case, $\sigma$ is clearly decreasing along $Q_1$, while in the second $\sigma$ decreases along $Q_2$.   Either way, we contradict $\dot{\sigma} \geq 0$.  Note that this argument establishes that $\Gamma_{x_1} \subset \partial G_1$ and, moreover, it shows that no branch of $u(S)$ can lie in an open sector $C_{u(x), -\eps_0}$, as defined in Step 1.

\begin{figure}
\centering
\begin{minipage}{.5\textwidth}
  \centering
  \psfragscanon
\psfrag{a}{$L_{u(x)}$}
\psfrag{b}{$L_{u(x_1)}$}
\psfrag{c}{$u(y)$}
\psfrag{d}{$\Gamma_x$}
\psfrag{e}{$u(x)$}
\psfrag{f}[r]{$u(x_1)$}
\psfrag{g}[r]{$u(w)$}
\psfrag{h}{$\scriptscriptstyle{G'\arrowvert  0}$}
\psfrag{i}{$\scriptscriptstyle{\nu}$}
\psfrag{j}{$\scriptscriptstyle{\tau}$}
\psfrag{k}{$\scriptscriptstyle{\nu}$}
\psfrag{l}{$\scriptscriptstyle{\tau}$}
\psfrag{Q}{$Q_2$}
\psfrag{0}{$0$}
\psfrag{r}{$\scriptscriptstyle{\nu}$}
\psfrag{t}{$\scriptscriptstyle{\tau}$}
\psfrag{q}{$Q_1$}  
  \includegraphics[width=.9\linewidth]{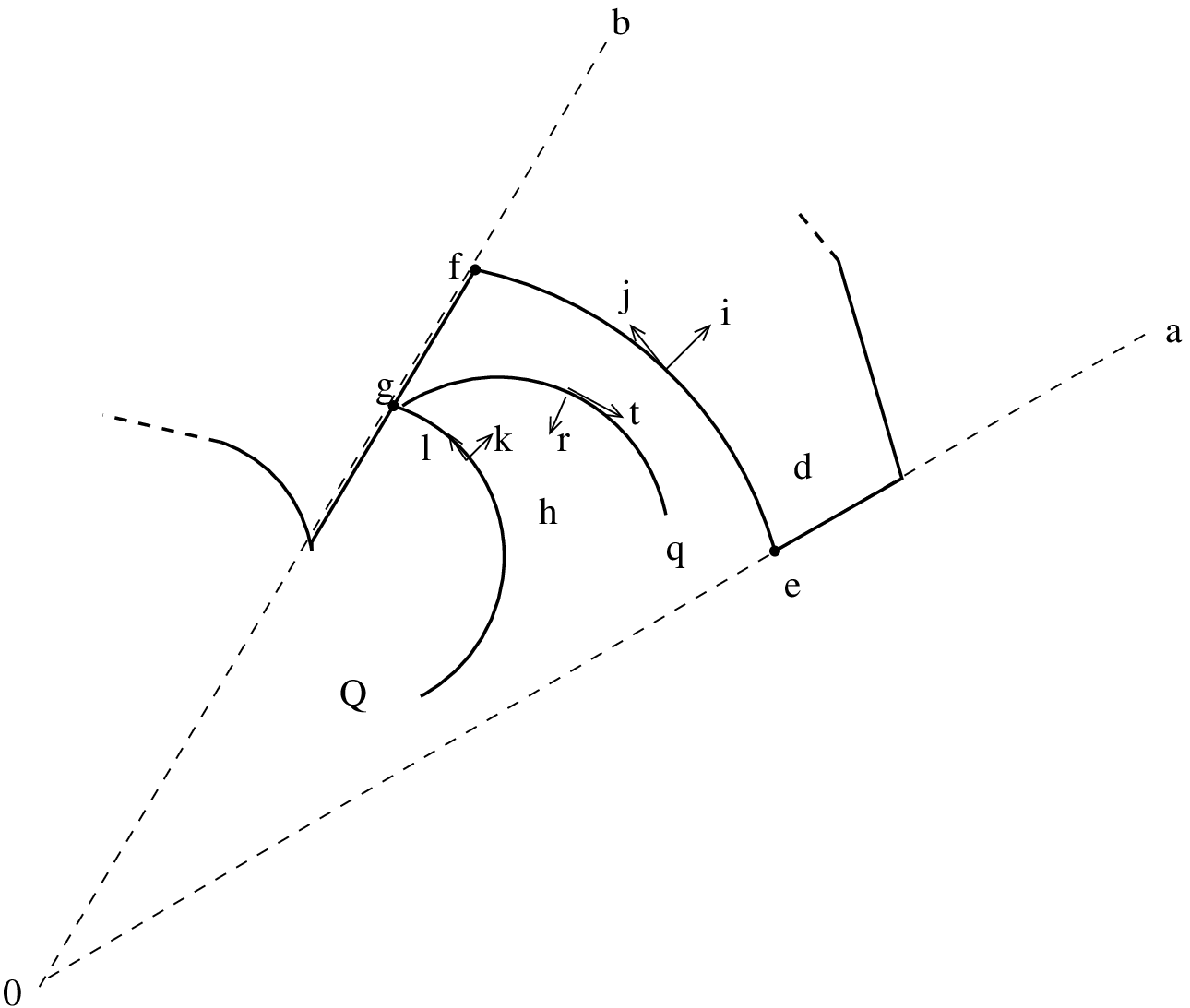}
  \caption{If $\restr{d}{G'}=0$ then $\sigma$ decreases along $Q_1$.}
  \label{Z2a}
\end{minipage}%
\begin{minipage}{.5\textwidth}
  \centering
\psfragscanon
\psfrag{a}{$L_{u(x)}$}
\psfrag{b}{$L_{u(x_1)}$}
\psfrag{c}{$u(y)$}
\psfrag{d}{$\Gamma_x$}
\psfrag{e}{$u(x)$}
\psfrag{f}[r]{$u(x_1)$}
\psfrag{g}[r]{$u(w)$}
\psfrag{h}{$\scriptscriptstyle{G'\arrowvert  1}$}
\psfrag{i}{$\scriptscriptstyle{\nu}$}
\psfrag{j}{$\scriptscriptstyle{\tau}$}
\psfrag{k}{$\scriptscriptstyle{\nu}$}
\psfrag{l}{$\scriptscriptstyle{\tau}$}
\psfrag{Q}{$Q_2$}
\psfrag{0}{$0$}
\psfrag{r}{$\scriptscriptstyle{\nu}$}
\psfrag{t}{$\scriptscriptstyle{\tau}$}
\psfrag{q}{$Q_1$}  
  \includegraphics[width=.9\linewidth]{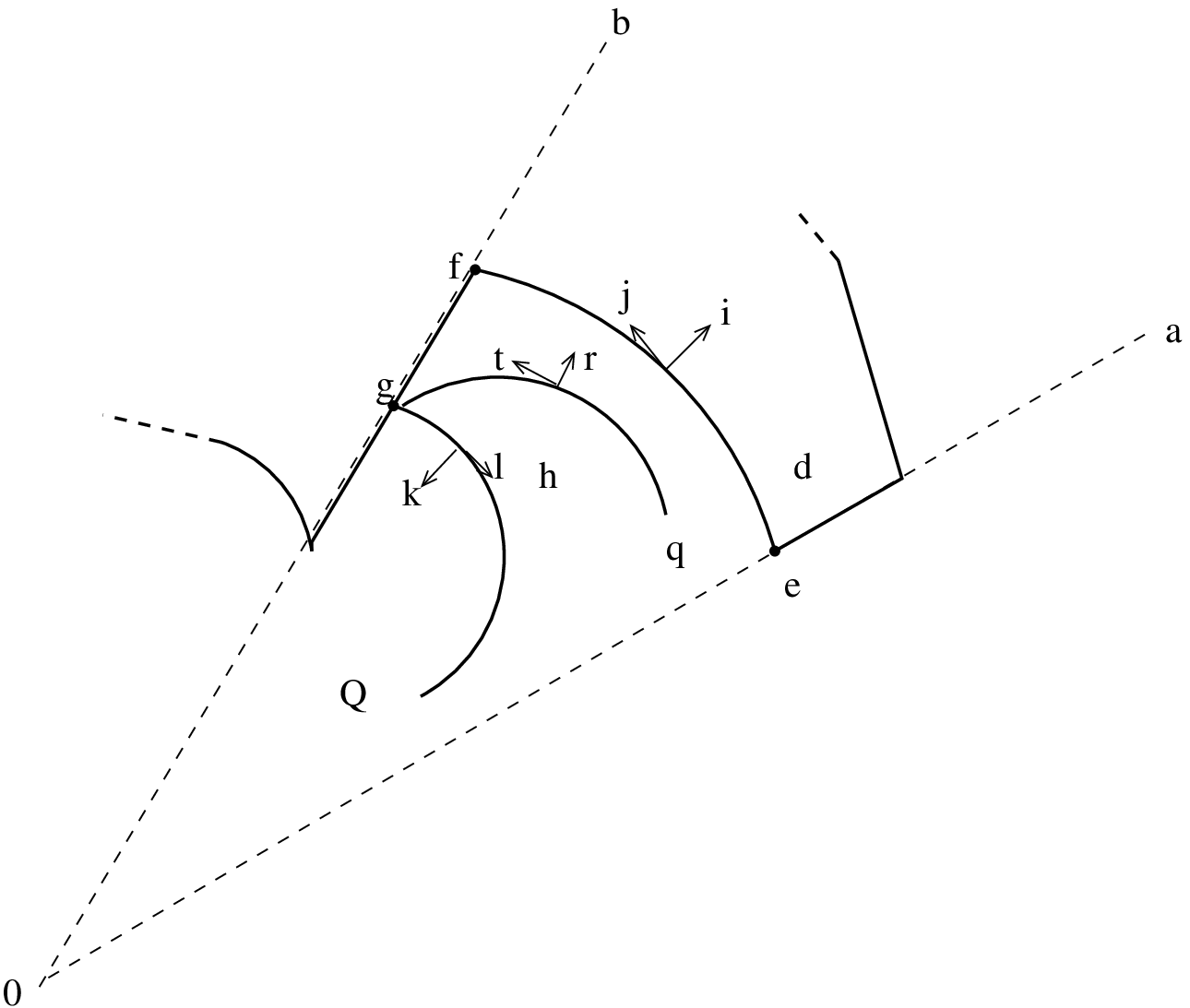}
  \caption{If $\restr{d}{G'}=1$ then $\sigma$ decreases along $Q_2$.}
  \label{Z2b}
\end{minipage}
\end{figure}

Define $u(x_2)$ as that point in $u(S) \cap \Gamma_{x_1}$ such that $|u(x_2)| \leq |z|$ for all $z \in \Gamma_{x_1}$.  Adjoin the interval $[u(x_2),u(x_1)]$ to $P_0$ and call the resulting curve $P_1$.  By the previous step, branches of $u(S)$ must leave $\Gamma_{x_1}$ by entering a cone of the form $C_{u(x_1),\eps_2}$, where $\eps_2 >0$.  Moreover, since $u: S \to u(S)$ is $1-1$ $\hone-$a.e., at least one of these branches must contain $u(x_2)$.  Therefore $P_1$ can be extended continuously from $u(x_2)$ by following the branch of $u(S)$ which pointwise minimizes its distance to $0$.   Iterating this process produces the desired curve $P(x)$, which by construction is contained in $\partial G_1$.
The method used to define $P(y)$ is so similar that we omit the details, apart from saying that the condition $P(y) \subset \partial G_1$ is not required to hold.

\vspace{1mm}
\noindent{\textbf{Step 3}}
Since $u(S)$ is connected it must be that $P(x)$ and $P(y)$ meet at some point $u(w)$, say.    Let $u(w)$ be the first such meeting point (relative to $x$ and $y$) in the sense that there is $\eps_{0} > 0$ such that  
\[\dist\left(L_{\Rot(-\eps)u(w)} \cap P(x), L_{\Rot(-\eps)u(w)} \cap P(y)\right) > 0 \ \ \ \ \forall \eps \in (0,\eps_{0}).\]
For definiteness, let $P(x)$ and $P(y)$ both terminate at $u(w)$.    A glance at Figure \ref{f1} may help to visualize the arrangement.  


  The degree $d$ is of compact support and is constant on each component, so it must in particular be that $\restr{d}{G_0}=0$ and $\restr{d}{G_{1}}=1$, where $G_0$ and $G_1$ were defined at the outset of Step 1.  Let $G$ be the component of $\R^2 \setminus u(S)$ such that for all sufficiently small $\gamma$ the sets $G \cap B(u(w),\gamma) \cap P(x)$ and $G \cap B(u(w),\gamma) \cap P(y)$ are nonempty.  The boundary of each component is contained in $u(S)$, which is of finite perimeter, so 
$\partial G$ is in particular $\hone$ measurable (and of finite $\hone$ measure).   By construction $\hone(\partial G \cap \partial G_1)>0$, so that necessarily $\restr{d}{G}=0$.   Moreover, $\hone(\partial G \cap P(y)) >0 $ together with $\restr{d}{G}=0$ implies there is a further component $G' \neq G_{0}$ such that $\restr{d}{G'}=1$ and for which $u(w) \in \partial G'$.   The local behaviour of the degree together with \eqref{j:tildenu} forces an orientation on the set $u(S)$ in a neighbourhood of $u(w)$, as shown in Fig \ref{f1}.  
\begin{figure}[ht]
\psfragscanon
\psfrag{g1}{$G \arrowvert 0$}
\psfrag{g2}{$G' \arrowvert 1$}
\psfrag{g3}{$G_1 \arrowvert 1$}
\psfrag{g}[r]{$u(w)$}
\psfrag{r}{$\scriptscriptstyle{\nu}$}
\psfrag{t}{$\scriptscriptstyle{\tau}$} 
\psfrag{py}{$P(y)$}
\psfrag{px}{$P(x)$}
\psfrag{L}{$L_{u(w)}$}
\psfrag{o}{$0$}
\includegraphics[width=0.6\textwidth]{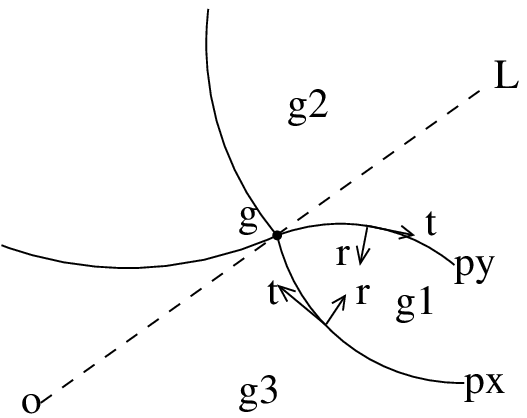}
\caption{Topological representation of the components of $\R^2 \setminus u(S)$ near the first meeting point $u(w)$ of $P(x)$ and $P(y)$.  The notation $G \, | \, n$ is used to indicate that the degree $d$ is $n$ on the component $G$.}
\label{f1}
\end{figure} 
Choosing $\omega:=[\alpha,\beta) \subset [0, 2\pi)$ so that $u(Re(\theta)) \in P(y)$ if $\theta \in \omega$, with $u(Re(\alpha))=u(w)$, it follows that $\sigma$ is decreasing on $\omega$, contradicting $\dot{\sigma} \geq 0$.   Thus part (ii) of Definition \ref{defstarshape} must hold, and we conclude that $u(S)$ is star-shaped.

\vspace{2mm}
\noindent($\mathbf{u(S)}$\textbf{ is star-shaped}$\ \Rightarrow \mathbf{t \geq 0}$).   Since part (i) of Definition \ref{defstarshape} holds we may assume that $\sigma$ has been chosen so that \eqref{j:sigmadef} and \eqref{j:sigmadot} apply.   Thus it is sufficient to show that $\dot{\sigma} \geq 0$ a.e. in $[0,2\pi)$.   Suppose for a contradiction that $\dot{\sigma} \geq 0$ does not hold a.e\frenchspacing. in $[0,2 \pi)$.    By reparametrizing we can suppose that $0$ is a Lebesgue point of $\dot{\sigma}$ and that $\dot{\sigma}(0) <0$.  Then $\sigma(\theta) < \sigma(0)$  if $\theta$ is sufficiently small and positive.    Now, either $\sigma$ is nonincreasing on the entire interval $[0,2\pi)$ or there is $\theta_{1} >0$ such that $\sigma(\theta) > \sigma(\theta_{1})$ if $\theta-\theta_{1}$ is sufficiently small and positive.   Without loss of generality we can suppose that $\theta_{1}$ is the first point in $(0,2 \pi)$ at which $\sigma$ fails to be nonincreasing.    In particular, for each sufficiently small $h$ there is $k(h)$ such that the sets 
\[\Lambda_{-h}  :=  L_{e(\sigma(\theta_{1}- h))} \cap u(S)\]
and
\[\Lambda_{k(h)}  :=  L_{e(\sigma(\theta_{1}+ k(h)))} \cap u(S)\]
are both contained in the same line.  According to part (ii) of Definition \ref{defstarshape}, the sets $L_{u(x)} \cap u(S)$ are connected for $\hone-$a.e. $x \in S$.   But this implies that $u$ fails to be $1-1$ on a set of positive $\hone$ measure, which  contradicts part (iv) of Proposition \ref{techprop1}.

Thus the only possibility is that $\sigma$ is nonincreasing on the whole interval $[0,2\pi)$.   But then the set $u(S)$ must be traversed clockwise with increasing $\theta$, and hence by \eqref{j:tildenu}, the generalized exterior normal points almost everywhere into the set $u(B(x_{0},R))$, which implies $\restr{d}{G_1}=0$, a contradiction.   Thus we conclude that $\dot{\sigma} \geq 0$ must hold almost everywhere.\qed
\end{proof}

\section{A variational inequality, positive twist and H\"{o}lder regularity}\label{prm}

In this section we introduce two functionals,  $E(u)$ and $F(u,\om',r')$:  the first measures the elastic stored energy of a deformation $u$ belonging to the class $\sca_1$ as defined in \eqref{a1}, while the second quantity has the key property that maps $u$ with $F(u,\om',r')=0$ have a.e. positive twist on balls of radius $r'$ centred at points of $\om'$. In view of the characterization of positive twist given in Section \ref{starshape}, $F(u,\om',r')=0$ encodes the condition that $u$ maps these balls to star-shaped sets on a `small scale'.    We regard $F(u,\om',r')=0$ as a condition which may or may not be satisfied rather than as a constraint to which all competing functions in $\sca_1$ are subjected\footnote{See Remark \ref{constrain} for more on this}.    The main result is that any local minimizer of $E$ such that $F(u,B(z,\delta),r')=0$ is H\"{o}lder continuous on any compact subset of $B(z,\delta) \subset \om$.    Thus, when it holds, the geometric condition involving star-shapedness translates into a means for proving regularity of the associated local minimizer of the energy $E$.  It does this by ensuring that (i) a certain useful class of variations is admissible, and (ii) the resulting variational inequality is suitably monotone in the sense that elliptic regularity theory can be applied to it.   Details of (i) and (ii) are given later in the section. 
\\[1mm]

\noindent {\bf{The functional}} $\mathbf{E(u)}$. 
It has previously been established that the stored energy of an elastic body occupying in a reference configuration the bounded domain $\om \subset \R^2$ can be expressed as
\begin{equation}\label{e1} E(u) = \int_{\om} W(\nabla u) \,dx. \end{equation}  
Here, the deformation $u$ belongs to $\sca_1$, as defined by \eqref{a1}.  We focus on a particular type of stored-energy function $W: \R \to [0,+\infty]$ given by
\[   W(A) = \lambda|A|^2 + h(\det A), \]
where $\lambda>0$ is a fixed constant and  the  function $h : \R \to [0,+\infty]$ satisfies 
\begin{align}\label{defh}
h(s) = & \left \{\begin{array}{l l} +\infty  & \ \textrm{if} \ s \leq 0  \\
|\ln s| & \ \textrm{if} \ s \in (0,c_1) \\
\theta(s)  & \ \textrm{if} \ s \in [c_1,c_2] \\
l s + m   & \ \textrm{if} \ s \in (c_2,+\infty).\end{array} \right.
\end{align}
The constants $l,m,c_1,c_2$ and the function $\theta: (c_1,c_2) \to (0,\infty)$ satisfy the following conditions:
\begin{itemize}
\item[(i)]  $0 < c_1 < 1 < c_2 < +\infty$;
\item[(ii)] $\theta \in C^{2}(c_1,c_2)$, with $\ddot{\theta} > 0$;
\item[(iii)] $\theta(c_1) = |\ln c_1|$ and $\theta(c_2)=l c_2 + m$;
\item[(iv)] $\dot{\theta}(c_1)=-\frac{1}{c_1}$ and $\dot{\theta}(c_2)= l$;
\item[(v)]  $\dot{\theta}(1)=0$ and $\theta(1)>0$.
\end{itemize}
In the region $s >0$ the main features of $h(s)$ are its logarithmic growth for $s$ small and positive, and linear growth for $s$ large and positive.   It is straightforward to check that when (i)-(v) hold the function $h$ is  $C^1$ and convex, with a unique global minimum of $\theta(1)$ at $1$.    The connecting function $\theta$ is of lesser importance, and there are many of these for which (i)-(v) hold.  For example, let $\theta_1$ be a given constant satisfying $\theta_1 > m+l$ and define
\begin{displaymath}\theta(s) = \left\{\begin{array}{l l} 1-\ln c_1 - \frac{s}{c_1} + \int_{c_1}^{s}(s-s')\psi_1(s')\,ds & \textrm{if} \ c_{1} \leq s \leq 1 \\
\theta_{1} + \int_{1}^{s}(s-s')\psi_2(s')\,ds &  \textrm{if} \ 1 \leq s \leq c_2.
\end{array}\right.\end{displaymath}
The positive, continuous functions $\psi_1$ and $\psi_2$ are subject to
\begin{align*} \int_{c_1}^{1} \psi_1(s') \,ds' & =  \frac{1}{c_1}  \\
 \int_{c_1}^{1} s'\psi_1(s') \,ds' & =  1-\theta_1 -\ln c_1 \\
\int_{1}^{c_2} \psi_2(s') \,ds' & =  l \\
 \int_{1}^{c_2} s'\psi_2(s') \,ds' &=  \theta_1 - m, 
\end{align*}
which can be satisfied by choosing $\psi_1$ and $\psi_2$ to be appropriate polynomials.\\[1mm]

We now define the class of admissible deformations by
\begin{equation}\label{defadmiss}
\sca:=\{u \in \sca_1: \ \ E(u) < +\infty\}.
\end{equation}

\noindent {\bf{The functional}} $\mathbf{F(u,\om',r')}$.
As advertised, $F$ will encode our notion of `positive twist'.   To motivate its construction, suppose for argument's sake that $\sca$ contains at least one diffeomorphism, $v$, say, where, in particular, $\det \nabla v > 0$ in $\om$.   By Proposition \ref{diffeo}, we can assert that for each $x_0$ in $\om$ there is some radius $r(x_0) >0$ such that 
\[t(x,x_0,v) \geq \frac{1}{2}\det \nabla v(x_0) \, |x-x_0|  \ \textrm{ if } x \in B(x_0,r(x_0)).\]
Here, $r(x_0) < \dist(x_0,\partial \om)$;  note also that because $v$ is $C^1$ we can suppose $r(\cdot)$ is bounded uniformly below on compact subsets $\om'$, say, of $\om$.   Let the constant $r'$ satisfy $0 < r' \leq \inf \{r(x_0): \ x_0 \in \om'\}$. 

Now let the function $g: \R \to [0,+\infty)$ be given by
\begin{displaymath}g(s) = \left\{\begin{array}{l l} -s & \textrm{ if } \ s \leq 0 \\ 0 & \textrm{ if } \ s \geq 0. \end{array}\right.
 \end{displaymath} 
In view of the inequality satisfied by $t(x,x_0,v)$ above, it is clear that
\[ F(u,\om',r'):=\int_{\om'}\int_{B(x_0,r')} g(t(x,x_0,u))\,dx \,dx_0\]
satisfies $F(u,\om',r')=0$.    The definition of $F$ below generalizes this idea to other admissible functions $u$.   Thus to each $u \in \sca_1$, each set $\om'$ which is compactly contained in $\om$ and each fixed $r' \in (0,\dist(\om',\partial \om))$ we associate the functional
  \begin{equation}\label{defF} F(u,\om',r') := \int_{x_0 \in \om'}\int_{B(x_0, r')} g(t,x,x_0,u)\,dx \,dx_0.\end{equation}
It can easily be checked that $g(t(x,x_0,u))$ is integrable on the set indicated, and that $F(u,\om',r') \geq 0$.   

In the next section we will consider local minimizers of the functional $E$.  To prove that at least one of these exists we appeal to the following well-known result concerning the existence of a  global minimizer of $E$ in $\sca$.   Since any global minimizer is in particular a local minimizer, the result clearly establishes the existence of the latter.

\begin{theorem}\emph{(}\cite[Theorem 6.1]{BM84}\emph{)}  Let the functional $E(u)$ be given by \eqref{e1} and $\sca$ by \eqref{defadmiss}.    Then there exists a minimizer of $E$ in $\sca$.
\end{theorem}

\subsection{A variational inequality} \label{elineq}

We begin by introducing a useful class of outer variations 
$\{u(x;x_0,\eps): \ \eps < 0, \ x_0 \in \om\}$ about a given $u$ in $\sca$.  The variations themselves are standard from the point of view of classical regularity theory in that they are of the form
\begin{equation}\label{defvar} \uex = u(x) + \eps \eta^2(x;x_0)(u(x)-u(x_0)),\end{equation}
where $\eps$ is a small parameter, the scalar function $\eta(\cdot;x_0)$ takes values in $[0,1]$ and is supported in a small ball about a given $x_0$ in $\om$.     The technical distinction needed in the elasticity setting is that we also require $\uex$ to belong to $\sca_1$, which necessarily implies that $\det \nabla_{x} \uex > 0$ must hold almost everywhere in $\om$.   The latter can be achieved by imposing on $u$ a condition of the form $F(u,B(z,\delta),r')=0$, where $B(z,r) \subset \om$, and by choosing the support of $\eta$ to be sufficiently small.   Indeed, if $t(x,x_0,u) \geq 0$ for $x \in B(x_0,r')$ and $x_0 \in B(z,\delta)$, it follows from the calculations given in Proposition \ref{p:basic-calc} that $\det \nabla_x \uex \geq \det \nabla_{x} u(x) /4 $ on all of $\om$ provided $\textrm{spt\,}\eta(\cdot,x_0)$ is contained in $B(x_0, r')$.  Hence, by properties of the stored-energy function $W$,  $E(u(\cdot,x_0,\eps)) < +\infty$.

Let us assume that 
\begin{equation}\label{Fzero}  F(u,B(z,\delta),r')=0\end{equation}
for some ball $B(z,\delta)$ and some $r' > 0$.   It follows that 
\begin{equation}\label{tpos} t(x,x_0,u) \geq 0 \ \textrm{ for a.e. } x \in B(x_0,r')  \textrm{ and a.e.} \ x_0 \in B(z,\delta).\end{equation} 
Fix $x_0 \in B(z,\delta)$ for which the inequality in \eqref{tpos} holds at a.e\frenchspacing. $x$ in $B(x_0,r')$.
Let $R=|x-x_0|$ and let $\eta(\cdot;x_0): \om \to [0,1]$ satisfy 
\begin{equation}\label{defeta} \eta(x,x_0)=f(R), \end{equation}
where $0 < 2r < r'$ and $f:[0,\infty) \to [0,1]$ is a smooth function satisfying
\begin{align}\label{F1} f(R)  &= \left\{\begin{array}{l l} 1 & \textrm{if} \  0 \leq R \leq r \\ 
0 &  \textrm{if} \   R \geq 2r \end{array}\right.
\end{align}
and, for some constant $c$ independent of $x_0$, $r$ and $R$,  
\begin{equation}\label{F2} |f'(R)|\leq c/r \ \ \textrm{for} \ r < R < 2r.\end{equation}   
We further assume that $f'(R) \leq 0$ for all $R \in (r,2r)$.   Set $f(R)=0$ if $R \geq 2r$.   In particular,
\begin{equation}\label{fprime}f'(R) \leq 0 \ \ \textrm{for all} \  0 <R < \infty.\end{equation}
Let $\eps <0$ and recall the definition of $\uex$ given in \eqref{defvar}, which in terms of $a:=u(x_0)$ reads
\begin{equation}\label{ueps}\uex= u(x) + \eps \eta^2(x,x_0) (u(x) - a).\end{equation}

To make the notation less cumbersome we will denote $\uex$ and $\eta(x;x_0)$ by $\ue(x)$ and $\eta(x)$ respectively.  We now gather together some expressions involving $\nabla \ue$, $\det\nabla \ue$ and $t(x,x_0,\ue)$. 
\begin{proposition}\label{p:basic-calc}Let $u$ belong to $\sca$, let $x_0$ be such that \eqref{Fzero} holds,  and define $\ue$ by \eqref{ueps}.   Then, with $R:=|x-x_0|$, $a=u(x_0)$, $e(x)=(x-x_0)/R$ and $-1/2 < \eps \leq 0$, 
\begin{align}\label{blaue} \nabla \ue & = (1+\eps \eta^2) \nabla u + \eps (u-a) \otimes \nabla (\eta^2), \\
 \label{detblaue1}
 \det \nabla \ue & = (1+\eps\eta^2)^2 \det \nabla u + 2 \eps f'(R)  \eta (1+\eps\eta^2) t(x,x_0,u)
\end{align}
In particular, 
\begin{align}\label{dquarter} \det \nabla u \geq \det \nabla \ue  & \geq \frac{\det \nabla u}{4} \ \ \textrm{a.e\frenchspacing. on} \  \om.\end{align} 
\end{proposition}
\begin{proof} 
A short calculation gives \eqref{blaue} and 
\[\det \nabla \ue  = (1+\eps\eta^2)^2 \det \nabla u + 2 \eps f'(R) \eta  (1+\eps\eta^2)\, \cof \nabla u \cdot \big((u-a) \otimes e\big).\]
Recalling \eqref{twist}, the term  $\cof \nabla u \cdot \big((u-a) \otimes e\big)$ in the latter is exactly $t(x,x_0,u)$, which gives \eqref{detblaue1}. 

To prove \eqref{dquarter} note that the assumptions \eqref{fprime}, $\eps \leq 0
$ and \eqref{Fzero} imply (via \eqref{tpos}) that  $\eps f'(R) t(x,x_0,u) \geq 0$ a.e\frenchspacing. on $B(x_0,2r)$.  In particular, the second term in \eqref{detblaue1} is a.e\frenchspacing. nonnegative, while the first is easily seen to be bounded below by $\det \nabla u/4$.   Thus \eqref{dquarter}.   \qed
\end{proof}

\begin{remark}\label{constrain}
It is possible to show that the functional $F(u,\om',r')$ is lower semicontinuous with respect to weak convergence in $W^{1,2}$ for fixed $\om'$ and $r'$ (see Lemma \ref{kswlsc}).   In particular, if $u^{(j)} \rightharpoonup u$ in $W^{1,2}$, if $F(u^{(j)},\om',r')=0$ for all $j$, where each $u^{(j)}$ belongs to $\sca$, and since $F$ is nonnegative, then $F(u,\om',r') =0$.   The existence of a minimizer $u$ of $E$ in the restricted class
\[\sca^{F}:=\{ u \in \sca_1: \ F(u,\om',r')=0\}\] 
is then, in conjunction with \cite[Theorem 6.1]{BM84}, not difficult to show.    We note that $\sca^F$ is nonempty provided one assumes that $\sca_1$ contains at least one diffeomorphism for which $F(v,\om',r')=0$:  see Proposition \ref{diffeo} and the construction of $F$ for the details.  As the results of Section \ref{starshape} show, the constraint $F(u,\om',r')=0$ translates into the condition that the minimizer of $E$ in $\sca^F$ maps sufficiently small circles centred at points $z$ of $\om'$ to star-shaped sets relative to $u(z)$.   The catch is that in order to exploit this minimality we require not only that $\ue \in \sca_1$, which is assured by $F(u,\om',r')=0$, but also that $F(\ue)=0$ in order that $\ue \in \sca^F$.    The latter condition does not seem to hold, and so one cannot conclude that $E(\ue) \geq E(u)$.  \end{remark}

 We now focus on deriving a variational inequality under the assumption that
the relative twist $t(x,x_0,u)$ is nonnegative a.e. on a ball about $x_0$.  As we have seen, this is certain to be the case for a.e\frenchspacing. $x_0$ in $B(z,\delta)$ provided $F(u,B(z,\delta),r')=0$ for some $r'>0$.  In the following we use the shorthand notation \[d_u (x):=\det \nabla u(x)\] and \[t(x):=t(x,x_0,u).\]   

\begin{lemma}\label{l:maxhprime}  Let the function $h: \R \to [0,+\infty]$ be given by \eqref{defh} and assume that $u \in \sca$.  
Assume futher than $t(x) \geq 0$ for a.e\frenchspacing. $x$ in $B(x_0,r')$ for some $r' > 0$.   Let $\ue$ be given by \eqref{ueps}, where $\eps <0$. Then $\ue \in \sca$,  
\begin{equation}\label{e:maxhprime} \max\{|h'(s)|: \ \ \min\{d_u (x), d_{\ue} (x)\} \leq s \leq 
 \max\{d_u (x), d_{\ue} (x)\}\} \leq \max\left\{l,\frac{4}{d_u}\right\},  
\end{equation}
and
\begin{equation}\label{ub2} \left|\frac{h(d_{\ue}) - h(d_{u})}{\eps}\right| \leq  \max\left\{l,\frac{4}{d_u (x)}\right\}\big((2\eta^2 +1/2)d_u + 2|f'(R)|t \eta \big),
\end{equation}
the right-hand side of \eqref{ub2} being independent of $\eps$.
\end{lemma}
\begin{proof}
The assertion that $\ue \in \sca$ involves showing that $\ue \in \sca_1$ and that $E(\ue)$ is finite.   The former holds easily; the latter can be checked by using \eqref{dquarter} and the definition of $h$ given in \eqref{defh}.

Let $m:=\min\{d_u (x), d_{\ue} (x)\}$ and  $M:=\max\{d_u (x), d_{\ue} (x)\}$.  To prove \eqref{e:maxhprime} we begin by observing that, by convexity, the maximum of $|h'|$ on the interval $[m,M]$ is either $|h'(m)|$ or $|h'(M)|$.    If $m<1$ then $|h'(m)| \leq \max\{ 1/d_u, 1/d_{\ue}\} <4 / d_u$, where the last inequality follows from \eqref{dquarter}.   If $M \leq 1$ then $|h'(M)| \leq |h'(m)| < 4/d_u$, while if $M\geq 1$ then $|h'(M)| \leq l$, and so \eqref{e:maxhprime} holds when $m<1$.  When $m\geq 1$ we easily have $0 \leq h'(m) \leq h'(M) \leq l$, and so \eqref{e:maxhprime} again holds.

Inequality \eqref{ub2} results from an application of \eqref{e:maxhprime} to the following:

\begin{align*}  \left|\frac{h(d_{\ue}) - h(d_{u})}{\eps}\right| & \leq \int_{m}^{M} \left|\frac{h'(s)}{\eps}\right| \,ds \\
& \leq  \max\left\{l,\frac{4}{d_u (x)}\right\}\left|\frac{M-m}{\eps}\right| \\
& \leq  \max\left\{l,\frac{4}{d_u (x)}\right\}\left( \left|\frac{(1+\eps\eta^2)^2 - 1}{\eps}\right|d_u
+ 2 \eta (1+\eps \eta^2)|f'(R)| t\right) \\
& \leq  \max\left\{l,\frac{4}{d_u (x)}\right\}\left((2\eta^2+1/2)d_u
+ 2 \eta |f'(R)| t \right).
\end{align*} \qed
\end{proof}

The preceding estimates are needed to calculate a bound on the quantity 
\[ \limsup_{\eps \to 0} \frac{E(\ue)-E(u)}{\eps} \] 
which appears in the variational principle set out in Theorem \ref{t1}.   

\begin{lemma}\label{ellipt1}  Let $E$ be defined by \eqref{e1} and let $u \in \sca$ be a strong local minimizer of $E$ in $\sca$ in the sense that there is $\gamma > 0$ such that 
\begin{equation}\label{wlm}  E(v) \geq E(u)  \ \ \forall \ v \in \sca \textrm{ s.t. } ||v-u||_{\infty;\om} < \gamma.\end{equation}
Assume that $F(u,B(z,\delta),r')=0$ and that $x_0 \in B(z,\delta)$ is such that $t(x) \geq 0$ for a.e. $x \in B(x_0,r')$. 
Then
\begin{align} \nonumber   \lambda \int_{\om} \left(\eta^2 |\nabla u|^2+ \nabla u \cdot ((u-a) \otimes \nabla (\eta^2)) \right) & \leq \int_{\{x \in \om: \ d_u (x) < 1\}} \eta^2 |h'(d_u)| d_u \,dx  \\
& \label{reg1} + \int_{\{x \in \om: \ d_u (x) \geq  1\}} \eta t h'(d_u) |f'(R)| \,dx \end{align}
where $x_0  \in \om$,  $\eta: \om \to [0,1]$ is given by $\eta(x):=f(|x-x_0|)$, and $f$ satisfies \eqref{F1}, \eqref{F2} and \eqref{fprime}
\end{lemma}

\begin{proof}   Let $\ue$ be defined by \eqref{ueps}.  Since $\ue$ is continuous and $\eta$ is bounded, it follows that $||\ue - u||_{\infty;\om} \to 0$ as $\eps \to 0$.   Applying 
\eqref{wlm}, we can assume that 
 \[ \frac{E(\ue)-E(u)}{\eps} \leq 0\]
for $\eps$ smaller in magnitude than $\min\{\gamma, 1/2\}$, and hence, since $\eps <0$, 
\begin{equation}\label{rk1} \lambda \int_{\om} \frac{|\nabla \ue|^2 - |\nabla u|^2}{\eps} \,dx + \int_{\om} \frac{h(d_{\ue})-h(d_u)}{\eps} \,dx \leq 0.
\end{equation} 
The derivation of 
\begin{equation}\label{yoyoma}\lim_{\eps \nearrow 0} \int_{\om} \frac{|\nabla \ue|^2 - |\nabla u|^2}{\eps} \,dx = 2\int_{\om} \eta^2 |\nabla u|^2+ \nabla u \cdot ((u-a) \otimes \nabla (\eta^2)) \,dx\end{equation} follows by applying \eqref{blaue} together with a suitable dominated convergence argument.  

The argument needed to derive the terms on the right-hand side of \eqref{reg1} is more delicate. One approach would be to apply a dominated convergence theorem in conjunction with an estimate such as \eqref{ub2}, but this could fail because for small values of $d_{u}(x)$, and recalling the definition of $t$ given in \eqref{twist}, the term on the right-hand side of \eqref{ub2} is potentially of order $|(\nabla u)^{-1}|$, which we cannot assume to be in $L^1 _{\textrm{loc}}(\om)$ .    However, it does apply on $\om_0$, where 
\[\om_0 :=\{x \in \om:  \ d_u (x) \geq \delta_0 \} \]
and where $\delta_0< c_1$ is a small, positive quantity to be chosen shortly.     Now, by \eqref{rk1}
\begin{align}\nonumber  2\lambda \int_{\om} \left(\eta^2 |\nabla u|^2+ \nabla u \cdot ((u-a) \otimes \nabla (\eta^2)) \right)  \leq &  -\frac{1}{\eps}\int_{\om_0} (h(d_{u^{\eps}})-  h(d_u)) \,dx  \\ 
\label{regsplit} &-\frac{1}{\eps}\int_{\om\setminus \om_0} (h(d_{u^{\eps}}) - h(d_u)) \,dx \end{align}  
By \eqref{ub2},
\[\left|  \frac{h(d_{u^{\eps}}) -h(d_{u})}{\eps} \right| \chi_{_{\om_0}} \leq \max\left\{l,\frac{4}{\delta_0}\right\}\left(   (2\eta^2 + 1/2)d_u +2|f'(R)| \eta t \right) \chi_{_{\om_0}} .\]
Hence
\begin{align}\nonumber \lim_{\eps \nearrow 0} -\frac{1}{\eps} \int_{\om_0} (h(d_{\ue})-h(d_u)) \,dx & = \int_{\om_0}-h'(d_u) \partial \arrowvert_{\eps=0} d_{\ue} \,dx \\ \label{omzero}
& = \int_{\om_0} -h'(d_u)\left(2 \eta^2 d_u + 2 f'(R) \eta t\right) dx.\end{align}
Splitting the range of integration in \eqref{omzero} into $\om_{0}^+:=\{x \in \om_0: \ d_u (x) \geq  1\}$ and $\om_{0}^- :=\{x \in \om_0 : \ d_u (x) <  1\}$, the integrand in \eqref{omzero} satisfies
\begin{align} (-h'(d_u))\left(2 \eta^2 d_u + 2 \eta t  f'(R)\right) & \leq \left\{ \begin{array}{l l} 2 \eta^2 |h'(d_u)|d_u & \textrm{if} \ x \in \om_0^- \\
2 l \eta t |f'(R)| &  \textrm{if} \ x \in \om_0^+. 
\end{array}\right.  \label{c1}
\end{align} 
The first inequality follows from the fact that $h'(d_u) <0$ if $d_u < 1$ and $\eta t f'(R) \leq 0$ by construction, so $-h'(d_u) \eta t f'(R) \leq 0$ on $\om_0^{-}$.  A similar argument using the fact that $0 \leq h'(d_u) \leq l$ on $\om_0^{+}$ yields the second inequality.  Notice that $\om_0^{-} = \{ x \in \om:  \ \delta_0 \leq  d_u (x) < 1\}$ and that $\om_0^{+} = \{x \in \om: d_u (x) \geq 1\}$.

The other term on the right-hand side of \eqref{regsplit} can be dealt with as follows.   Let $\om_1:=\om \setminus \om_0$, define
\begin{align}\label{defphi} \varphi(x;\eps) & :=  \frac{2\eps f'(R) \eta(x) t(x)}{(1+\eps \eta^{2}(x))}\end{align}
and notice that, by \eqref{detblaue1},  
\[d_{u^{\eps}}=(1+\eps \eta^2)^2 (d_u + \varphi).\] 
We now write $\om_1 = \om_2(\eps) \cup \om_2(\eps) \cup \om_4(\eps)$, where
\begin{align*} \om_2(\eps) & = \{ x \in \om_1: \ d_{u^{\eps}} \leq c_1\}, \\
\om_3(\eps) & = \{ x \in \om_1: \ c_1 < d_{u^{\eps}} < 1 \}, \\
\om_4(\eps) & = \{ x \in \om_1: \ d_{u^{\eps}} \geq 1\}.
\end{align*}
The constant $c_1 < 1$ is introduced in \eqref{defh}, according to which if $x$ belongs to $\om_2(\eps)$ then
 \begin{align}\label{norah}-\frac{1}{\eps}(h(d_{u^\eps})-h(d_u)) & = \frac{2}{\eps} \ln (1+\eps \eta^2) +\frac{\ln (d_u + \varphi) - \ln d_u}{\eps} \leq 2\eta^2 +\frac{|\eps|}{(1+\eps)},  \end{align}
where the last inequality holds because $\varphi \geq 0>\eps$, so that $(\ln (d_u+\varphi) - \ln d_u )/\eps \leq 0$, and by elementary estimates for $\ln(1+\eps \eta^2)$.   It follows that
\begin{align}\label{c2}\limsup_{\eps \nearrow 0} \int_{\om_{2}(\eps)}\frac{h(d_{u^\eps})-h(d_u)}{|\eps|}\,dx   \leq\int_{\om_{2}(0)} 2 \eta^2 \,dx. \end{align}
But $\om_{2}(0) = \{x \in \om: \ d_u (x) < \min\{c_1,\delta_0\}\}$, which, because $\delta_0 < c_1$, coincides with the set $\om_1$.

 It remains to consider the sets $\om_{3}(\eps)$ and $\om_{4}(\eps)$.   

\vspace{2mm}
\noindent{\textbf{Claim:}} For $j=3$ and $4$, $\mathcal{L}^2(\om_{j}(\eps)) \to 0$ as $\eps \to 0$.

\vspace{1mm}
\noindent{\textbf{Proof of the claim:}}   Let $c \geq c_1$ and define the set $\om(\eps,c):=\{x \in \om_1: \ d_{u^\eps} \geq c\}$.   Note that $\om_4(\eps) = \om(\eps,1) \subset \om(\eps,c_1)$ and $\om_3(\eps) \subset  \om(\eps,c_1)$, so  it suffices to prove that $\mathcal{L}^{2}(\om(\eps,c_1)) \to 0$ as $\eps \to 0$.  

For any $x \in \om(\eps,c_1)$ it holds that
\[ d_u + \varphi \geq \frac{c_1}{(1+\eps \eta^2)^2}  \geq c_1,\]
so that, on using the definition of $\varphi$ given in \eqref{defphi}, the fact that $d_u <  \delta_0$ on $\om_1$, and the assumptions $|\eps| < 1/2$ and $\eta^2 \leq 1$,  
\[ 2 \eps f'(R) \eta t \geq \frac{c_1-\delta_{0}}{2} \ \ \textrm{if} \ x \in \om(\eps,c_1).\]
Hence 
\[\mathcal{L}^2(\om(\eps,c_1))  \leq \frac{4|\eps|}{(c_1-\delta_0)}\int_{\om(\eps,c_1)} |f'(R)|\eta t \,dx.  \]
By the definition of $t$ given in \eqref{twist} and the fact that $f'$ has support in a fixed annulus about $x_0$, the integrand $|f'(R)|\eta t$ is clearly in $L^{2}(\om)$.  The claim now follows.

\vspace{1mm}  The next step in the proof of the proposition consists in showing that 
\begin{align*}\limsup_{\eps \nearrow 0} \int_{\om_{j}(\eps)} \frac{h(d_{u^\eps})-h(d_u)}{|\eps|}\,dx  \leq 0  \end{align*}
when $j=3$ and $4$.  The argument needed in the case of $\om_3(\eps)$ is as follows.

Let $x \in \om_3(\eps)$, so that $c_1 < d_{u^\eps} < 1$.   Note that 
$\lim_{\eps \to 0} d_{u^\eps}(x) = d_u(x) < \delta_0 < c_1$, and hence, since $\tilde{\eps} \mapsto d_{u^{\tilde{\eps}}}$ is quadratic in $\tilde{\eps}$, there exists a unique $\eps_1(x)$, say, measurable in $x$, belonging to the interval $(\eps,0)$ and such that $d_{u^{\eps_1(x)}}(x) = c_1$, with 
$d_{u^{\tilde{\eps}}}(x) > c_1$ if $\eps < \tilde{\eps} < \eps_1(x)$ and
$d_{u^{\tilde{\eps}}}(x) < c_1$ if $\eps_1(x) < \tilde{\eps} < 0$.   Now write
\begin{align*}-\frac{1}{\eps}(h(d_{u^\eps})-h(d_u)) & =  -\frac{1}{\eps}(h(d_{u^\eps})-h(d_{u^{\eps_1(x)}}))  -\frac{1}{\eps}(h(d_{u^{\eps_1(x)}})-h(d_u)) \\
& =: A + B.\ \end{align*}
The quotient $A$ can be estimated by writing
\begin{align*} A & = \int_{\eps}^{\eps_1(x)} \frac{2\theta'(d_{u^{\tilde{\eps}}})}{\eps}\left[(1+\tilde{\eps}\eta^2) d_u + (1+ 2\tilde{\eps}\eta^2)f'(R)\eta t\right]\,d\tilde{\eps}
\\
& \leq 2|\theta'(c_1)|\left(1- \frac{\eps_1(x)}{\eps}\right)\left[\delta_0 + |f'(R)|t(x)\right].
\end{align*}
The right-hand side of the last line above is clearly an integrable function, so that, by standard results together with the claim above, $\int_{\om_{3}(\eps)} A(x) \,dx \to 0$ as $\eps \to 0$.   $B$ can be estimated similarly, this time by appealing to the expressions in \eqref{norah} but with $\eps_1(x)$ in place of $\eps$, followed by an integration over $\om_3(\eps)$ and another application of the claim above.

\vspace{1mm}
Next, we consider the integral of the difference quotient over $\om_4(\eps)$.  Note that because $1+\eps \eta^2 < 1$ we have $d_{u^\eps} < d_u + \varphi$, and therefore since $1 \leq d_{u^{\eps}}$ and $h$ is increasing on $(1,+\infty)$, 
\begin{align*}\int_{\om_4(\eps)} \frac{h(d_{u^\eps}) - h(d_u)}{|\eps|}& \,dx  \leq \int_{\om_4(\eps)} \frac{h(d_u+\varphi)-h(d_u)}{|\eps|}\,dx \\
& \leq \int_{\om_4(\eps)\cap \{d_u+\varphi < c_2\}} \frac{(h(c_2)+\ln \delta_0)}{|\eps|} \,dx + \\ & + \int_{\om_4(\eps) \cap \{d_u+\varphi \geq c_2\}} \frac{(l \delta_0+ m + \ln \delta_0)}{|\eps|} \,dx + \int_{\om_4(\eps)} \frac{l \varphi}{|\eps|}\,dx.  
\end{align*}
Let $\delta_0$ be so small that the first two integrands are negative.   The term in $\varphi$ satisfies
\[\int_{\om_4(\eps)} \frac{l\varphi}{|\eps|}\,dx \leq  \int_{\om_4(\eps)} 4 l |f'(R)|\eta t\,dx,\]
where the  integrand $4l|f'(R)|\eta t \chi_{_{\om_1^{+} (\eps)}}$ converges a.e. and boundedly in $L^{2}(\om)$ to zero (by the claim above).   We have therefore shown that 
\begin{align}\label{c3}\limsup_{\eps \nearrow 0}\int_{\om_4(\eps)} \frac{h(d_{u^\eps}) - h(d_u)}{|\eps|} \, dx  \leq  0.
\end{align}     
To finally obtain \eqref{reg1}, take the limit $\eps \to 0$ in \eqref{rk1} and apply \eqref{yoyoma}, \eqref{omzero}, \eqref{c1}, \eqref{c2} and \eqref{c3}.  This gives
\begin{align}
\label{lacets}
 \lambda \int_{\om} \left(\eta^2 |\nabla u|^2+ \nabla u \cdot ((u-a) \otimes \nabla (\eta^2)) \right) & \leq \int_{\om_0^{-}} \eta^2 \,dx + \int_{\om_0^{+}} l \eta t |f'(R)| \,dx  \\ \nonumber & + \int_{\om_1} \eta^2\,dx.
\end{align}
The remarks above imply that 
\[ \om_0^{-} \cup \om_1 = \{x \in \om: \ d_u (x) < 1\}\]
and we have already pointed out that $\om_0 ^{+}$ is the set on which $d_u(x) \geq 1$.  Hence \eqref{lacets} implies \eqref{reg1} when the definition of $h$ given in \eqref{defh} is applied.\qed
\end{proof}

\subsection{H\"{o}lder regularity} \label{holreg}

We now apply some well-known steps from elliptic regularity theory to prove that $u$ satisfying $\eqref{reg1}$ must be H\"{o}lder continuous.  The argument uses the following technical lemma together with a version of Morrey's Dirichlet growth inequality.  The lemma could be deduced from \cite[Lemma 2.1, Chapter 3]{Gi83}; however, to keep the paper self-contained we give a direct proof below.

\begin{lemma}\label{l:growth}Let $c, r_1 >0$, $\mu \in (0,1)$ and $p\geq 1$.  Suppose $\phi: (0,2r_1) \to (0,\infty)$ is nondecreasing and satisfies
\begin{equation}\label{iterant} \phi(r) \leq c r^p + \mu \phi (2r) \ \ \textrm{for all} \  r \in (0,r_1).
\end{equation}
Then there is $\alpha >0$ and  constants $\eps_0 >0$ and $c>0$ depending on $r_1$ such that 
\begin{equation}\label{decay}  \phi(r) \leq c r^\alpha  \ \textrm{for all} \  r \in (0,\eps_0).\end{equation}
\end{lemma}
\begin{proof} Fix $\alpha'>0$ so that $\mu = 2^{-\alpha'}$, let $r \in (0,r_1)$ and let the integer $k$ be such that $2^{-(k+1)}r_1 < r \leq 2^{-k}r_1$.  Since $\phi$ is nondecreasing, $\phi(r) \leq \phi(2^{-k}r_1)$.   Next, iterating the expression in \eqref{iterant} gives 
\begin{align}\label{bingo} \phi(2^{-k}r_{1})&  \leq  c \left\{(2^{-k} r_1)^{p}+ \mu(2^{-(k-1)} r_1)^{p} + \ldots + \mu^{(k-1)}(2^{-1}r_1)^{p}\right\} + \mu^{k}\phi(r_1). 
\end{align}
Let $k' = \lfloor k/2 \rfloor$.  The sum on the right-hand side of the last equation can then be written as
\begin{align*} c \left\{(2^{-k} r_1)^{p}+ \ldots + \mu^{(k-1)}(2^{-1}r_1)^{p}\right\} + \mu^{k}\phi(r_1) & = c \sum_{j=0}^{k'} \mu^{j}(2^{-(k-j)} r_1)^{p} \\ & + c \sum_{j=k'+1}^{p} \mu^{j}(2^{-(k-j)} r_1)^{p}  + \mu^{k} \phi(r_1).
\end{align*}
Estimating each sum, we obtain
\begin{align}\label{sum1} \sum_{j=0}^{k'} \mu^j (2^{-(k-j)} r_1)^{p} & \leq \frac{(2^{-(k-k')}r_1)^{p}}{1-\mu} \\
\label{sum2}\sum_{j=k'+1}^{k-1} \mu^j (2^{-(k-j)} r_1)^{p} & \leq \frac{\mu^{k'+1}(2^{-1} r_{1})^{p}}{1-\mu}
\end{align}
Since $k'+1 \geq (k+1)/2$ and $2^{-(k+1)}r_1 < r$, we have $\mu^{k'+1} < (r/r_{1})^{\alpha'/2}$, so that the right-hand side of \eqref{sum2} is bounded above by the quantity
$(2^{-1}r)^{p}(r/r_{1})^{\alpha'/2}$.

To deal with the term in \eqref{sum1} first note that $k/2 \leq k - k' \leq (k+1)/2$, so that $2^{-(k-k')}r_1 \leq 2^{-k/2}r_{1}$.   Since $2^{-(k+1)}r_1 < r$ we must have $(2^{-k/2}r_1)^{1/2} < (2rr_1)^{1/2}$, and hence $(2^{-(k-k')}r_1)^{p} < (2rr_1)^{p/2}$.

In summary, \eqref{bingo}, together with \eqref{sum1} and \eqref{sum2}, yields
\[ \phi(r) \leq \tilde{c}(p,r_1)\max\{r^{\frac{p}{2}}, r^{\frac{\alpha'}{2}}\}+2^{\alpha'}\left( \frac{r}{r_1}\right)^{\alpha'} \phi(r_{1})\] 
for some constant $\tilde{c}$ with dependence as shown.   The statement of the lemma now follows.\qed
\end{proof}

\begin{theorem}\label{t1}  Let $E$ be defined by \eqref{e1} and suppose that $u$ is a strong local minimizer of $E$ in $\sca$.   Suppose further that $F(u,B(z,\delta),r')=0$ for some $B(z,\delta) \subset \om$ and $r'>0$, where $F$ is given by \eqref{defF}.
Then $u$ is H\"{o}lder continuous on any compact subset of $B(z,\delta)$.

\end{theorem}

\begin{proof}  By Lemma \ref{ellipt1}, we can assume that for almost every $x_0$ in $B(z,\delta)$ the inequality \eqref{reg1} holds.   The right-hand side of \eqref{reg1} consists of two terms which satisfy the following estimates:
\begin{align*}  \int_{\{x \in \om: \ d_u (x) < 1\}} \eta^2 |h'(d_u)| d_u \,dx & \leq \frac{1}{c_1}\int_{B_{2r}} \eta^2 \,dx  \\
 \int_{\{x \in \om: \ d_u (x) \geq  1\}} \eta t h'(d_u) |f'(R)| \,dx  & \leq l \int_{B_{2r}} \eta |\nabla \eta||\nabla u||u-a| \,dx.
\end{align*}
The first estimate uses the definition of $h$ given in \eqref{defh};   the second uses the fact that $\eta(x)=f(|x-x_0|)$, where $f$ has support in $[0,2r]$, along with \eqref{twist}.   Substituting these into \eqref{reg1} gives
\begin{align*}  \lambda \int_{\om} \left(\eta^2 |\nabla u|^2+ \nabla u \cdot ((u-a) \otimes \nabla (\eta^2)) \right) & \leq \frac{1}{c_1}\int_{B_{2r}} \eta^2 \,dx  +\\+ & \frac{l}{2} \int_{B_{2r}}  \eta |\nabla \eta||\nabla u||u-a| \,dx + Cr^2.
\end{align*}
The term involving $\nabla u \cdot ((u-a) \otimes \nabla (\eta^2))$ is clearly bounded pointwise by a multiple of $\eta|\nabla \eta||\nabla u||u-a|$, so that on rearranging we obtain for some constant $C$ that
\begin{align*}  \lambda \int_{\om} \eta^2 |\nabla u|^2 \, dx  \leq &\frac{1}{c_1}\int_{B_{2r}} \eta^2 \,dx  + C \int_{B_{2r}}  \eta |\nabla \eta||\nabla u||u-a| \,dx \\
 \leq &\frac{1}{c_1}\int_{B_{2r}} \eta^2 \,dx  +\frac{\lambda}{2} \int_{B_{2r}}\eta^2 |\nabla u|^2 \,dx + \\ &  \quad \quad \quad \quad \quad \quad +
\frac{C^2}{2\lambda} \int_{B_{2r}\setminus B_r} |\nabla \eta|^2 |u-a|^2\,dx.
\end{align*}

By the Poincar\'{e} inequality \eqref{taptap123} and the bound $|\nabla \eta| \leq c/r$, 
\[  \int_{B_{2r}\setminus B_r} |\nabla \eta|^2 |u-a|^2 \,dx \leq \frac{7c^2}{3} \int_{B_{2r} \setminus B_{r}} |\nabla u|^2\,dx,\] 
so that by gathering terms in $\eta^2|\nabla u|^2$ and applying \eqref{F1} we obtain
\begin{align*}
\int_{B_r} |\nabla u|^2 \,dx \leq C' r^2 + C' \int_{B_{2r} \setminus B_r} |\nabla u|^2 \,dx,
\end{align*}
where $C'$ depends on $c_1$, $\lambda$ and $l$ but not on $x_0$.  Applying Widman's hole-filling technique, we add $C'\int_{B_r} |\nabla u|^2 \,dx$ to both sides and divide by $C'+1$, giving
\[ \int_{B_r} |\nabla u|^2 \,dx \leq \mu r^2 + \mu \int_{B_{2r}} |\nabla u|^2 \,dx,\]   
with $\mu:=C'/(C'+1)$.  Let 
\[ \phi(x_0, r):=\int_{B(x_0,r)} |\nabla u|^2 \,dx\]
and apply Lemma \ref{l:growth} to deduce that $\phi(x_0,r) \leq c r^\alpha$ for some $\alpha > 0$ which is independent of $r$.   Thus for almost every $x_0$ in $B(z,\delta)$ it holds that
\begin{equation}\label{dg}\int_{B(x_0,r)}|\nabla u|^2 \,dx  \leq c r^\alpha \ \ \textrm{for } 0 < r < r'.\end{equation} 
Since $x_0 \mapsto \phi(x_0,r)$ is continuous, it follows that \eqref{dg} holds at all $x_0$ in $B(z,\delta)$.   The desired H\"{o}lder continuity of $u$ is now an immediate consequence of Morrey's Dirichlet growth theorem (\cite[Theorem 3.5.2]{Mo66}).   
\end{proof}

\section{H\"{o}lder regular shear solutions}\label{shear}

In this section we exhibit a variational problem from nonlinear elasticity theory to which some of the ideas described earlier in the paper apply.    Our purpose is to directly illustrate that the singular term involving $h(\det \nabla u)$ appearing in the energy functional has a regularizing effect on the minimizer.

We consider a restricted class of deformations, which we call shear maps, and which take the form
\[ \us(x) =  \left(\begin{array}{c}x_{1} \\ x_{2}+\sigma(x)\end{array}\right).\] 
Here, $\sigma$ is a scalar-valued map which belongs to a certain subclass of the space $W^{1,2}(Q;\R)$ and $Q$ is the set $[-1,1]^2$ in $\R^2$.  The choice of $Q$ is not pivotal, but it does enable the problem to be visualised more easily.   If we let $l_{x_{1}}=\{(x_{1}, t)^{T}: \ t \in \R\}$  be the vertical line through $(x_{1},0)^{T}$ then $\us$ maps  $Q \cap l_{x_{1}}$ to a subset of $l_{x_{1}}$ for each $-1< x_{1} < 1$.  Thus nearby lines in $Q$ are `sheared' relative to one another.

An advantage of using shear maps $\us$ is that the Jacobian $\det \nabla \us$ is affine in $\sigma_{,_{2}}$.   Specifically, a short calculation shows that
\begin{equation}\label{detaff}
\det \nabla \us = 1+\sigma_{,_{2}}.
\end{equation}
Thus, in the notation introduced in Section \ref{prm}, the energy $E(\us)$ of a shear map takes the form
\begin{equation}\label{eshear}
E(\us) = \int_{Q} \lambda|\1 + e_2 \otimes \nabla \sigma|^2 + h(1+\sigma_{,_{2}})\,dx.
\end{equation}
The term involving $h(1+\sigma_{,_{2}})$ has a regularizing effect on $\sigma$ in the $x_2$-direction only.  Therefore, it seems to be necessary to assume some extra regularity in the $x_1$ direction.    We impose the condition that for some $M$ 
\begin{equation}\label{lipshear}|\sigma(x_1,x_2) - \sigma(x_{1}',x_2)| \leq M |x_1 - x_1'| \ \ \ \ \forall \ (x_1,x_2), (x_{1}',x_2) \in Q,\end{equation} 
which clearly amounts to assuming that $\sigma(x_1,x_2)$ is Lipschitz in the $x_1$-direction with Lipschitz constant independent of $x_2$.	We conjecture that this can be weakened to a uniform H\"{older} assumption on $\sigma$ in the $x_1$ direction.

The form of the Jacobian \eqref{detaff} for shear maps means that we can treat a wider class of stored-energy functions than those detailed in Section \ref{prm}, where we can in particular allow quadratic growth of $h(s)$ as $s \to +\infty$ and the behaviour of $h(s)$ as $s \to 0+$ need not be logarithmic in $s$.    See Section \ref{shear1} for the assumptions applied to $h$ in the shear map setting.    The affine Jacobian \eqref{detaff}    also allows us to design outer variations $u_{\sigma^{\eps}}$, say, within the class of shear maps which obey the constraint $\det \nabla u_{\sigma^{\eps}} > 0$ a.e.   In previous sections of the paper this was achieved by imposing a condition of positive twist;  here, we avoid imposing that condition, which is hard to verify, and replace it with the simpler condition \eqref{lipshear} in the $x_1$-direction.

\subsection{The functional $E(u)$} 	\label{shear1}
	
In light of the comments above, we modify the energy functional $E(u)$ introduced in Section \ref{prm}.  Let $W: \R^{2 \times 2} \to [0,+\infty]$ be given by
\begin{equation}\label{w2} W(A) = \lambda |A|^2 + h(\det A)  \ \ \ A \in \R^{2 \times 2}, \end{equation}
where
\begin{itemize}\item[(H0)] $h:\R \to [0,+\infty]$ is $C^{1}$ and strictly convex, 
\item[(H1)] $\lim_{s \to 0+} h(s) = +\infty$; 
\item[(H2)]  $h'(1)=0$ and $h(s) = +\infty$ if $s \leq 0$;
\item[(H3)] there are constants $q_{1}>0$ and $q_2$ such that 
$h'(s) \leq 2 q_{1}s  + q_{2}$ for all $s > 1$; 
\item[(H4)] there is $K>0$ and $s_0>0$ such that $h(s/2) \leq K h(s)$ for all $s \in (0,s_0)$.
\end{itemize}
Finally, let 
\begin{equation}\label{ee}E(u) = \int_{Q} W(\nabla u(x))\,dx. \end{equation}

We require admissible shear maps $\us$ to satisfy $E(\us) < +\infty$ together with an appropriate boundary condition, which may be applied on part of $\partial Q$ only if desired.    Let us fix $\sigma_0$ such that $u_{\sigma_0}$ satisfies $E(u_{\sigma_0})< +\infty$.   There are many such functions:  for example, if $\varphi_{\pm}(x_1)$ are both Lipschitz on $[-1,1]$ then the function 
\[ \sigma_{0}(x_1,x_2):=\left(\frac{1+x_2}{2}\right) \varphi_{+}(x_1) + 
\left(\frac{1-x_2}{2}\right) \varphi_{-}(x_1)\]
is such that $u_{\sigma_{0}}$ belongs to $W^{1,2}(Q,\R^2)$, with
\[ \det \nabla u_{\sigma_{0}} = 1+ \frac{1}{2}(\varphi_{+}(x_1)-\varphi_{-}(x_1)).\]
By further imposing $\varphi_{+}(x_1) \geq \varphi_{-}(x_1)$ for $-1\leq x_1 \leq 1$ we ensure that $\det \nabla u_{\sigma_0} \geq 1$, from which $E(u_{\sigma_{0}}) < +\infty$ follows easily. 

\begin{definition}\label{defas}
Let $u_{\sigma_{0}}$ be such that $E(u_{\sigma_0})< +\infty$ and let $(\partial Q)_{D} \subset \partial Q$.  Define the class $\sca_{s}$ of admissible shear maps as follows:
\[ \sca_s:=\{ \us(x)=x + \sigma(x) e_2 \in W^{1,2}(Q,\R^2): \, E(\us) < +\infty, \, \us = u_{\sigma_0} \ \textrm{on} \ (\partial Q)_D\}.\]
Here, $\us = u_{\sigma_0}$ on $(\partial Q)_{D}$ in the sense of trace, and we have suppressed the dependence of $\sca_s$ on $\sigma_{0}$ and $(\partial Q)_{D}$.
\end{definition}

The existence of a minimizing shear map $\us$ in $\sca_{s}$ now follows from methods established in \cite{BM84}.   

\begin{proposition} Let $E$ be given by \eqref{ee} and $\sca_s$ be as per Definition \ref{defas}.     Then there exists a minimizer of $E$ in $\sca_{s}$. 
\end{proposition} 

\begin{proof}  By \cite[Theorem 6.1]{BM84}, the integrand $W$ defined in \eqref{w2} subject to assumptions (H0)-(H3) is such that $E$ is sequentially lower semicontinuous with respect to weak convergence in $W^{1,2}$.  Moreover, it is clear that if $u_{\sigma^{(j)}}$ is a minimizing sequence for $E$ in $\sca_{s}$ satisfying  $u_{\sigma^{(j)}} \rightharpoonup v$ for some $v \in W^{1,2}(Q,\R^{2})$ then $v$ must also be a shear map, that is $v=\us$ for some $\sigma$ whose trace on $(\partial Q)_{D}$ agrees with that of $u_{\sigma_0}$.  Hence $\us$ minimizes $E$ in $\sca_s$. \qed
\end{proof}

The next result is a technical lemma which uses condition \eqref{lipshear} and the determinant constraint $\det \nabla \us > 0$ a.e. to generate one-sided bounds on $\sigma$.   We use the notation $x_0=(x_{01},x_{02})^{T}$.   
\begin{lemma}\label{outerspace} Let $\us$ belong to $\sca_s$ and suppose $\sigma$ satisfies condition \eqref{lipshear}.   Let $r > 0$ and suppose $x_0=(x_{01},x_{02})^{T}$ is such that $B(x_0,r) \subset Q$.    Then there exists $C>0$ such that for almost every $x_0$ 
\begin{equation}\label{b1}\sigma(x) - \sigma(x_0) \leq C|x-x_0|\end{equation} 
for almost every $x \in B(x_0,r)$ such that $x_2 < x_{02}$, and 
\begin{equation}\label{b2} \sigma(x) -\sigma(x_0) \geq -C|x-x_0| \end{equation}
for almost every $x \in B(x_0,r)$ such that $x_2 > x_{02}$.
\end{lemma}
\begin{proof} First note that $E(\us)<+\infty$ implies $\det \nabla \us >0$ almost everywhere.  In view of \eqref{detaff}, this gives $\sigma_{,_{2}} > -1$ almost everywhere.
Let $x_2 < x_{02}$.  Since $\sigma$ is absolutely continuous along almost all lines, we can suppose that for almost every $x_0$ it is the case that
\begin{align*} \sigma(x_1,x_2) & = -\int_{x_2}^{x_{02}} \sigma_{,_{2}}(x_1, s)\,ds + \sigma(x_1,x_{02})  \\& < x_{02}-x_2 + \sigma(x_1,x_{02})\\
& \leq |x-x_0| + |\sigma(x_1,x_{02})-\sigma(x_{01},x_{02})|+\sigma(x_{01},x_{02}) \\
& \leq |x-x_0| + M|x_{1}-x_{01}| + \sigma(x_{01},x_{02}) \\
& \leq (1+M)|x-x_0| + \sigma(x_{01},x_{02}).
\end{align*}
Hence \eqref{b1} holds with $C=1+M$.  The proof of \eqref{b2} can be obtained by exchanging $x$ and $x_0$ in the above, so that the same constant $C$ suffices for both inequalities.\qed
\end{proof}

\begin{remark} Note that \eqref{b1}, which follows from the assumption $\sigma_{,_{2}}>-1$, gives an upper \emph{but not lower} bound on $\sigma(x)-\sigma(x_0)$ in the region $x_2<x_{02}$.  The lower bound will follow from an application of elliptic regularity theory, as we shall see.     Similar comments apply to the inequality \eqref{b2}.\end{remark}

The previous result together with assumption \eqref{lipshear} now allows us to construct variations $u_{\sigma^{\eps}}$ about $\us$ which obey the constraint $\det \nabla u_{\sigma^{\eps}}>0$.  We continue to use the notation $R:=|x-x_0|$.
\begin{proposition}\label{outervar} Let $\us$ belong to $\sca_s$ and assume that \eqref{lipshear} holds.  Let $B(x_0,2r) \subset Q$.  Let
$\eta: Q \to \R$ be a smooth cut-off function such that \[\eta(x)=f(|x-x_0|),\] where $f$ satisfies 
\eqref{F1}, \eqref{F2} and \eqref{fprime}.  Define for each $\eps<0$ the functions
\begin{equation}\label{pm}\sigma^{\eps,\pm}(x):= \sigma(x) + \eps \eta^2(x) (\sigma(x)-\sigma(x_0)\mp C R)^{\pm},\end{equation}
where the constant $C$ is chosen so that 
\begin{align}\label{b3} \sigma(x)-\sigma(x_0) - C R \leq 0 & \quad \textrm{if} \ x_2 <x_{02},\\
\label{b4} \sigma(x) -\sigma(x_0) + C R \geq 0 &  \quad \textrm{if} \ x_2 > x_{02}.
\end{align}
Then for almost every such $x_0$ the shear maps $u_{\sigma^{\eps,\pm}}$ belong to $\sca_s$ and they satisfy
\begin{equation}\label{detbounds} \frac{1}{2} \det \nabla \us \leq \det \nabla u_{\sigma^{\eps,\pm}} \leq \det \nabla \us + C_{\pm}' \eps \end{equation}
where the constants $C_{\pm}'=1+C+\frac{c}{r}||\sigma -\sigma(x_0) \mp CR||_{\infty; B(x_0, 2r)}$. 
\end{proposition}

\begin{proof}  By Lemma \ref{outerspace}, there is a constant $C$ such that \eqref{b3} and \eqref{b4} hold.   Consider 
\[ \sigma^{\eps,+}(x) = \sigma(x) + \eps \eta^2 (\sigma(x)-\sigma(x_0)-CR)^{+}. \]
It is weakly differentiable and the Jacobian associated with the shear map $u_{\sigma^{\eps,+}}$ is
\begin{align} \nonumber \det \nabla u_{\sigma^{\eps,+}}  =1+\sigma_{,_{2}}(x) + & 2 \eps f(R)f'(R)(\nabla R)_{2}(\sigma(x)- \sigma(x_0)-CR)^{+}+ \\  \label{b5} & \quad \quad \quad  + \eps \eta^2 \sigma_{,_{2}}(x)  - C \eps \eta^2 (\nabla R)_{2} \end{align}
if $\sigma(x)-\sigma(x_0)-CR>0$ and
\[ \det \nabla u_{\sigma^{\eps,+}} = 1+\sigma_{,_{2}}(x)\] otherwise.
In the latter case there is nothing to show.  In the former,  note that 
\[ (\nabla R)_2 = \frac{x_2 - x_{02}}{R}\]
is positive for a.e. $x$ such that $\sigma(x)-\sigma(x_0)-CR>0$.  This is because, by Lemma \ref{outerspace},  $\sigma(x)-\sigma(x_0)-CR \leq 0$ for a.e. $x$ such that $x_2 <x_{02}$, so that, in short, $(\nabla R)_2\xi_+ \geq 0 $ a.e.  .    Slightly rearranging the expression for \eqref{b5} gives
\begin{align}\label{b6} \det \nabla u_{\sigma^{\eps,+}} & =(1+\eps\eta^2)(1+\sigma_{,_{2}}) - \eps \eta^2 + \\ & \nonumber   + 2 |\eps| f(R)|f'(R)|(\nabla R)_{2}(\sigma(x)-\sigma(x_0)-CR)^{+} +C |\eps| \eta^2 (\nabla R)_{2}\end{align}
on the set where $\sigma(x)-\sigma(x_0)-CR > 0$.    Provided $-\frac{1}{2} < \eps < 0$, the first term satisfies 
\[ (1+\eps\eta^2)(1+\sigma_{,_{2}}) \geq \frac{1}{2} \det \nabla \us,\]
and for any $\eps>0$ all other terms are nonnegative.  Hence the lower bound in \eqref{detbounds}.

The upper bound in \eqref{detbounds} follows from \eqref{b6} once the inequalities $|f'(R)| \leq c/r$ and $|f(R)|$ are recalled.   The argument needed for the map $\sigma^{\eps,-}$ is so similar that it is omitted.

By inequality \eqref{detbounds} and hypotheses (H3) and (H4), it follows that
\[\int_{Q} h(\det \nabla u_{\sigma^{\eps,\pm}})\, dx< +\infty\] 
for all $\eps \in (-1/2,0)$.  It is straighforward to check that $\nabla u_{\sigma^{\eps,\pm}}$ belongs to $L^2(Q)$, and so $E(u_{\sigma^{\eps,\pm}})<+\infty$.  Since $\eta$ has compact support in $Q$ it follows that $u_{\sigma^{\eps,\pm}}=u_{\sigma_{0}}$ on $(\partial Q)_{D}$.    Thus, for such $\eps$, all maps $u_{\sigma^{\eps,\pm}}$ belong to $\sca_s$. \qed
\end{proof}

We are now in a position to form a variational inequality which in fact applies to a local minimizer of the energy $E(\cdot)$ in the class $\sca_{s}$.  The following result is analogous to Lemma \ref{ellipt1}.

\begin{lemma}\label{ellipt2}  Let $E$ be defined by \eqref{ee} and let $\us \in \sca_{s}$ be a strong local minimizer of $E$ in $\sca_s$ in the sense that there is $\gamma > 0$ such that 
\begin{equation}\label{slmshear}  E(v) \geq E(\us)  \ \ \forall \ v \in \sca_{s} \textrm{ s.t. } ||v-\us||_{\infty;Q} < \gamma.\end{equation}   Suppose $\sigma$ satisfies condition \eqref{lipshear}.   Then there is a constant $C>0$ such that
\begin{align}\nonumber
& \int_{Q} \lambda ((\nabla (\eta^2))_{2}+\nabla \sigma \cdot \nabla (\eta^2)) \xi_{\pm}   +  \lambda \eta^2 (\sigma_{,_{2}}  \mp C(\nabla R)_{2})\chi_{\pm}\,dx + \\ \nonumber & + \int_{Q} \lambda \eta^2 \nabla \sigma \cdot (\nabla \sigma \mp C \nabla R) \chi_{+} \, dx \\
\label{b10} & \leq \int_{\{x \in Q : \ \det \nabla \us > 1\}}\frac{|h'(\det \nabla \us)|}{2} (|\nabla (\eta^2)|\xi_{\pm} + C\eta^2)\chi_{\pm} \,dx,
\end{align}
where $\eta$ is a smooth cut-off function satisfying \eqref{defeta}, \eqref{F1}, \eqref{F2} and \eqref{fprime}, $\xi_{\pm}(x):=(\sigma(x)-\sigma(x_0) \mp CR)^{\pm}$, and $\chi_{\pm}$ is the indicator function of the set $\{y \in Q: \ \xi_{\pm}(y) \neq 0\}$.
\end{lemma}

\begin{proof}  Define $u_{\sigma^{\eps,\pm}}$ as in Proposition \ref{outervar} and note that $u_{\sigma^{\eps,\pm}} \to \us$ in $W^{1,\infty}(Q,\R^2)$ as $\eps \to 0$.  In particular, since $u_{\sigma^{\eps,\pm}}$ belong to $\sca_s$ and $\us$ is by hypothesis a local minimizer in that class, it follows from \eqref{slmshear} that
\begin{equation}\label{b7} \frac{E(u_{\sigma^{\eps,\pm}})-E(\us)}{\eps} \leq 0\end{equation}
for all sufficiently small $\eps <0$.  

Let us focus on the variations $u_{\sigma^{\eps,+}}$. 
The quotient in \eqref{b7} consists of two terms, the first of which obeys
\begin{align}\nonumber \lim_{\eps \to 0}\int_{Q} & \frac{\lambda (|\nabla u_{\sigma^{\eps,+}}|^{2}-|\nabla \us|^2)}{2\eps}\,dx  = \int_{Q}  ((\nabla (\eta^2))_{2}+\nabla \sigma \cdot \nabla (\eta^2)) \xi_{\pm} \,dx + \\ \label{b11} & + \int_{Q} \lambda \eta^2 (\sigma_{,_{2}} - C(\nabla R)_{2})\chi_{+} + \lambda \eta^2 \nabla \sigma \cdot (\nabla \sigma - C \nabla R) \chi_{+} \, dx.\end{align}

The second term in the quotient \eqref{b7} essentially involves taking the derivative of the energy with respect to the singular term $\int_{Q} h(\det u_{\sigma^{\eps,+}})\,dx$.  To do this we use ideas from the proof of Lemma \ref{ellipt1}, beginning by rewriting \eqref{b7} as 
\begin{equation}\label{r1} \int_{Q} \frac{\lambda (|\nabla u_{\sigma^{\eps,+}}|^{2}-|\nabla \us|^2)}{\eps}\,dx  \leq - \int_{Q}  \frac{(h(\det u_{\sigma^{\eps,+}})-h(\det \nabla \us))}{\eps}\,dx.\end{equation}

  Let $\delta_0 >0$ be a small  positive constant to be chosen later and define
\[ Q_{0}^+ = \{ x \in Q: \det \nabla \us > \delta_0\}.\]
Now, by \eqref{b5}, $\det \nabla u_{\sigma^{\eps,+}}=\det \nabla \us + \eps \zeta$
where 
\begin{equation}\label{z1} \zeta(x):= (\nabla (\eta^2))_{2}\xi_{+} + \eta^2 (\sigma_{,_{2}}(x) - C (\nabla R)_{2})\chi_+. \end{equation}
A dominated convergence argument then yields
\begin{align*}
\limsup_{\eps \nearrow 0} \int_{Q_{0}^{+}} & \frac{-(h(\det u_{\sigma^{\eps,+}})-h(\det \nabla \us))}{\eps}\,dx = \int_{Q_0^+}-h'(\det \nabla \us) \zeta \, dx \\
& \hspace{30mm}= \int_{\{x \in Q_0^+: \ \det \nabla \us >1\}} -h'(\det \nabla \us) \zeta \,dx + \\
& \hspace{30mm}\  + \int_{\{x \in Q_0^+: \ \det \nabla \us \leq 1\}} -h'(\det \nabla \us) \zeta \,dx.  \end{align*}
Noting that $-h'(\det \nabla \us) \geq 0$ when $\det \nabla \us \leq 1$, and that $\sigma_{,_{2}}<0$ on the same set, it follows that $-h'(\det \nabla \us) \zeta(x) \leq 0$ on $\{x \in Q_0^+: \ \det \nabla \us \leq 1\}$.  (Here one uses that $(\nabla R)_2 \chi_+ \geq 0$.)  Thus the last integral above is bounded above by $0$.   In view of \eqref{z1}, the fact that $h$ is increasing on $[1,+\infty)$ and since $\sigma_{,_{2}} >0$ when $\det \nabla \us > 1$,  the penultimate integral is bounded above by 
\[ \int_{\{x \in Q_0^+ : \ \det \nabla \us > 1\}} |h'(\det \nabla \us)|(|\nabla(\eta^2))|\xi_+ + C\eta^2)\chi_+ \,dx.\]  
In summary,
\begin{align}\nonumber  \limsup_{\eps \nearrow 0} & \int_{Q_{0}^{+}} \frac{-(h(\det u_{\sigma^{\eps,+}})-h(\det \nabla \us))}{\eps} \,dx \\ \label{b12} &\leq  \int_{\{x \in Q : \ \det \nabla \us > 1\}}|h'(\det \nabla \us)| (|\nabla (\eta^2)|\xi_+ + C\eta^2)\chi_+ \,dx.\end{align}

Now let
\[ Q_{0}^- = \{ x \in Q: \ \det \nabla \us \leq \delta_0\}\]
and write 
\begin{align*}\int_{Q_0^-}\frac{h(d+ \eps \zeta) - h(d)}{|\eps|} \,dx & =\int_{Q_0^- \cap\{ d+\eps \zeta <1\}}\frac{h(d + \eps \zeta) - h(d)}{|\eps|} \,dx \\ & + \int_{Q_0^- \cap \{d +\eps \zeta \geq 1\}}\frac{h(d + \eps \zeta) - h(d)}{|\eps|} \,dx,\end{align*}
where $d:= \det \nabla \us$ and $Q_0^- \cap \{d+\eps \zeta < 1\}$ is shorthand for $\{x \in Q_0^-: \ d(x) + \eps \zeta(x) < 1\}$.   Notice that $d+\eps \zeta \geq  d$ on the set where $d < \delta_0$ (since $\zeta (x) \leq 0$ there, as can be seen by inspecting \eqref{z1}), and so, because $h(s)$ is decreasing for $0 < s < 1$, it follows that 
\[ \frac{h(d+\eps \zeta) - h(d)}{|\eps|} \leq 0 \ \ \textrm{if} \  x \in \{x \in Q_0^-: \ d(x) + \eps \zeta(x) < 1\}.\] 

To estimate the second integral we first use (H3) to deduce that there is a constant $q_3$ such that
\[h(s) \leq q_1 s^2 + q_2 s + q_3 \ \ \textrm{if } s \geq 1.\] 
This, when coupled with the fact that that if $d <\delta_0<1$ then $-h(d) \leq -h(\delta_0)$, gives
\begin{align}\nonumber \int_{Q_0^- \cap \{d +\eps \zeta \geq 1\}}\frac{h(d + \eps \zeta) - h(d)}{|\eps|}\,dx & \leq \int_{Q_0^- \cap \{d +\eps \zeta \geq 1\}} \!\!\!\!\frac{q_1 \delta_0^2+q_2^+\delta_0+q_3-h(\delta_0)}{|\eps|}\,dx \\
&\label{b8} - \int_{Q_0^- \cap \{d +\eps \zeta \geq 1\}} (2q_1d+q_2) \zeta + \eps q_1 \zeta^2\,dx. 
\end{align}
By applying (H1), the first integral on the right-hand side of \eqref{b8} is bounded above by zero provided $\delta_0$ is chosen so small that
\[ q_1 \delta_0^2+q_2^+\delta_0+q_3-h(\delta_0) \leq 0.\]

The second integral on the right-hand side of \eqref{b8} satisfies 
\begin{align}\label{b9}
\limsup_{\eps \nearrow 0} -\int_{Q_0^- \cap \{d +\eps \zeta \geq 1\}} (2q_1d+q_2) \zeta + \eps q_1 \zeta^2 \,dx & \leq 0.
\end{align}
To see this we let $S_{\eps}:=\{x \in Q_0^-: \ d(x) +\eps \zeta(x) \geq 1\}$ and note that
since $\eps \zeta(x) \geq (1-\delta_0)$ on $S_{\eps}$ and $\zeta$ belongs to $L^{2}(Q)$, it must be that $\scl^{2}(S_{\eps}) \to 0$ as $\eps \to 0$.  Hence, using H\"{o}lder's inequality, for example, together with $\zeta \in L^2(Q)$, 
\begin{align*} \int_{S_{\eps}} \zeta \,dx & \to 0 \ \ \textrm{as } \eps \to 0.\end{align*}
The fact that $\zeta$ belongs to $L^{2}(Q)$ also clearly implies that $\int_{S_{\eps}} \eps \zeta^2 \,dx$ tends to zero as $\eps \to 0$.  Hence, when we recall that $d \leq \delta_0$ on the set over which we are integrating, \eqref{b9} holds, from which it follows that
\begin{align}\label{b14}\limsup_{\eps \nearrow 0} \int_{Q_0^-}\frac{h(\det \nabla \us + \eps \zeta) - h(\det \nabla \us)}{|\eps|} \,dx & \leq 0.\end{align}
Finally, \eqref{b10} follows from \eqref{r1} together with \eqref{b11}, \eqref{b12} and \eqref{b14}.   \qed
 \end{proof}

The preceding lemma now allows us to improve the regularity of $\sigma$ in the $x_2$ direction, as follows.

\begin{theorem}  Let $E$ be defined by \eqref{ee} and let $\us$ be a strong local minimizer of the functional $E$ (in the sense of \eqref{slmshear}) in the class $\sca_s$.  Assume that condition \eqref{lipshear} holds.  Then $\us$ is locally H\"{o}lder continuous in the set $Q$.
\end{theorem}

\begin{proof}

The proof has two steps.\\[1mm]

\noindent {\bf Step 1}  Let $x_0$ belong to $B(z,\delta) \subset Q$ and choose $\delta$ small enough that the conditions of Lemma \ref{ellipt2} apply, and hence, in particular that \eqref{b10} holds.   Reusing the notation of Lemma \ref{ellipt2},  rearranging \eqref{b10} and applying standard estimates yields
\begin{align*}\int_{Q} \lambda \eta^2 |\nabla \sigma|^2 \chi_+ \,dx  & \leq \int_{Q}   \lambda \eta^2 ((1+C)|\nabla \sigma|  + C)\chi_{+} +  \lambda |\nabla(\eta^2)|(1+|\nabla \sigma|)\xi_+ \,dx \\&   + 
\int_{\{x \in Q : \ \det \nabla \us > 1\}}\frac{|h'(\det \nabla \us)|}{2} (|\nabla(\eta^2)|\xi_+ + C\eta^2)\chi_+ \,dx.
\end{align*}
Now, by (H3), 
\[ |h'(\det \nabla \us)| \leq 2q_1 + q_2^+ + 2q_1 |\nabla \sigma| \ \textrm{if } \det \nabla \us \geq 1.\]
This therefore gives
\begin{align*}\int_{Q} \lambda \eta^2 |\nabla \sigma|^2 \chi_+ \,dx & \leq \int_{Q}   \lambda\eta^2((1+C)|\nabla \sigma|  + C)\chi_{+} +  \lambda |\nabla(\eta^2)|(1+|\nabla \sigma|)\xi_+ \,dx \\&   + 
\int_{\{x \in Q : \ \det \nabla \us > 1\}}(q_4 + q_1 |\nabla \sigma|)(|\nabla(\eta^2)|\xi_+ + C\eta^2)\chi_+ \,dx,
\end{align*}
where $q_4:=q_1 + q_2^+/2 $.
By repeatedly applying the inequality $2ab \leq k^2 a^2 + (b/k)^2$ for suitable choices of $a$, $b$ and $k$, it is straightforward to show that there are constants $B_1$ and $B_2$ depending only on $\lambda$, $C$ and $q_1$, $q_2$ and $q_4$ such that
\begin{align*}\int_{Q} \eta^2 |\nabla \sigma|^2 \chi_+ \,dx & \leq \int_{Q} B_1 \eta^2  \chi_+  + B_2 (\eta^2 + |\nabla \eta|^2)(\xi_+)^2 + \frac{\eta^2}{2} |\nabla \sigma|^2\,\chi_+ dx.  \end{align*}
Applying the bound $|\nabla \eta| \leq c/r$ on the support $B(x_0,2r)\setminus B(x_0,r)$ of $ \nabla \eta$ and the fact that $\xi_+$ is bounded (since it is continuous), we obtain
\begin{align}\label{b16}\int_{Q} \frac{\eta^2}{2}  |\nabla \sigma|^2 \chi_+ \,dx & \leq B_1 r^2  + \frac{B_2}{r^2}\int_{B(x_0,2r)\setminus B(x_0,r)}(\xi_+)^2  \,dx,\end{align}
where $B_1$ and $B_2$ are possibly different constants, a convention we maintain for the rest of the proof.    Lemma \ref{app1} applies to the integral on the right-hand side of \eqref{b16}, with the result that  
\[ \int_{B(x_0,r)}|\nabla \sigma|^2 \chi_+ \,dx \leq B_1 r^2 + B_2 \int_{B(x_0,2r)\setminus B(x_0,r)} |\nabla \sigma|^2 \chi_+ + C^2 \chi_+ \,dx.\]
Note that we have used the easily verified fact that $|\nabla \xi_{+}|^2 \leq 2(|\nabla \sigma|^2 + C^2) \chi_{+}$.  Again employing Widman's hole-filling technique (that is, adding $B_2 \int_{B(x_0,r)}|\nabla \sigma|^2 \chi_+\,dx$ to both sides) we obtain
\begin{align}\phi_+(r) &\leq B_{1}r^2 + \mu \phi_+(2r)\end{align}
where $\mu : =B_2 /(B_2 + 1)$ and $\phi_+(r):=\int_{B(x_0,r)} |\nabla \sigma|^2 \chi_+ \,dx$. By Lemma \ref{l:growth}, there is a constant $\alpha >0$ depending on the constants $B_1$ and $B_2$ and $r'>0$ independent of $x_0$ such that $\phi_+(r) \leq B_3 r^\alpha$ for all $r<r'$.  A similar argument with $\xi_{-}$ in place of $\xi_+$ and $\chi_{-}$ in place of $\chi_{+}$ yields $\phi_{-}(r) \leq B_3 r^{\alpha}$ for $r<r'$, where $\phi_{-}(r) = \int_{B(x_0,r)}|\nabla \sigma|^2 \chi_{-} \,dx$. 
\\[2mm]
{\bf Step 2} Since $|\nabla \xi_{\pm}|^2 \leq 2(|\nabla \sigma|^2 + C^2) \chi_{\pm}$, it follows from Step 1 that there is a constant $c>0$ independent of $x_0$ and $r$ such that the functions $\xi_{\pm}$ satisfy  
\begin{align}\label{xiplusminus}   \int_{B(x_0,r)} |\nabla \xi_{\pm}|^2 \,dx \leq c r^\alpha
\end{align}
for all $r<r'$ and almost every $x_0$ in $B(z,\delta)$.  Arguing as in the proof of Theorem \ref{t1}, \eqref{xiplusminus} holds for all $x_0$ in $B(z,\delta)$, and hence by Morrey's Dirichlet growth theorem, (\cite[Theorem 3.5.2]{Mo66}), each function $(\sigma(x)-\sigma(x_0) \mp CR)^{\pm}$ is locally H\"{o}lder continuous.    Let $R<r'$ and consider $\sigma(x)-\sigma(x_0)-CR$.  If this quantity is positive then the H\"{o}lder continuity of $\xi_{+}$ at $x_0$ gives
\[|\sigma(x)-\sigma(x_0)-CR|\leq cR^{\alpha}\]
and hence
\begin{equation}\label{holdersigmaplus}|\sigma(x) - \sigma(x_0)| \leq c' R^{\alpha}.\end{equation}
If $\sigma(x)-\sigma(x_0)-CR \leq 0$ then there are two possibilities depending on whether  $\sigma(x)-\sigma(x_0)+CR \geq 0$ or $\sigma(x)-\sigma(x_0)+CR < 0$.  If $\sigma(x)-\sigma(x_0)+CR \geq 0$ then 
\[    -CR \leq \sigma(x)-\sigma(x_0) \leq CR, \]
and so certainly \eqref{holdersigmaplus} again holds (since $\alpha \leq 1$).   If $\sigma(x)-\sigma(x_0) + CR <0$ then the H\"{o}lder continuity of $\xi_{-}$ at $x_0$ gives 
\[|\sigma(x)-\sigma(x_0)+ CR|\leq cR^{\alpha}\]
and hence
\begin{equation}\label{holdersigmaminus}|\sigma(x) - \sigma(x_0)| \leq c' R^{\alpha}.
\end{equation}
Hence $\sigma$ is H\"{o}lder continuous on compact subsets of $Q$, which immediately implies the same property for $\us$. \qed
\end{proof}

\setcounter{section}{0}
\setcounter{equation}{0}
\renewcommand{\theequation}{\thesection.\arabic{equation}}
\setcounter{definition}{0}

\appendix
\section*{Appendix}
\renewcommand{\thesection}{A} 

The following result is implicit in the proof of \cite[Lemma 1, Section 4.5.2]{EG92}.    We give a short proof here because the Poincar\'{e} inequality used in Theorem \ref{t1} above depends critically upon it.

\begin{lemma}\label{app1}  Let $u \in W^{1,2}(B(x_0,3r),\R^2)$ be continuous, where $B(x_0,2r)$ is the ball in $\R^{2}$ of radius $2r$ and centre $x_0$.  Then 
\begin{equation}\label{taptap}\int_{B(x_0,2r)\setminus B(x_0,r)}|u(w)-u(x_0)|^2\,dw \leq \frac{7 r^3}{3}\int_{B(x_0,2r)\setminus B(x_0,r)}|\nabla u(w)|^{2}|w-x_0|^{-1}\,dw.\end{equation}    In particular,
\begin{equation}\label{taptap123} 
\int_{B(x_0,2r)\setminus B(x_0,r)}|u(w)-u(x_0)|^2\,dw \leq \frac{7 r^2}{3}\int_{B(x_0,2r)\setminus B(x_0,r)}|\nabla u(w)|^{2}\,dw.
\end{equation}

\end{lemma}  

\begin{proof} Let $A:=B(x_0,2r)\setminus B(x_0,r)$.  First note that since $|w-x_0|^{-1}$ is bounded on $A$, and since $u$ is continuous by hypothesis, it is enough to prove \eqref{taptap} for $C^1$ functions and then use a standard approximation argument.      Following \cite{EG92}, we obtain, for $s >0$,
\begin{equation}\label{preint}\int_{A \cap \partial B(x_0,s)} |u(y)-u(x_0)|^2\,d\sch^1(y) \leq s^{2} \int_{A \cap B(x_0,s)} |\nabla u(w)|^2 |w-x_0|^{-1}\, dw.\end{equation}
For $s$ in the range $(r,2r)$ the right-hand side is bounded above by 
\[ s^2 \int_{A}  |\nabla u(w)|^2 |w-x_0|^{-1}\, dw.\]
Hence, by integrating \eqref{preint} over $(r,2r)$, inequality \eqref{taptap} follows.   Inequality \eqref{taptap123} is now immediate from the observation that $|w-x_0| \geq r$ if $w \in A$. \qed
\end{proof}

\begin{lemma}\label{kswlsc} The functional $u \mapsto F(u, \om',r')$ is sequentially weakly lower semicontinuous in $W^{1,2}(\om,\R^2)$ for fixed $\om'$ and $r'$.\end{lemma}
\begin{proof}   Let 
\[\hat{\om}:= \bigcup_{x_0 \in \om'}(\{x_0\} \times B(x_0,r'))\]
and suppose $u^{(j)}$ converges weakly in $W^{1,2}(\om,\R^2)$ to $u$.   H\"{o}lder's inequality gives
\[\int_{\hat{\om}} |t(x,x_0,u^{(j)})|\,dx\,dx_0 \leq c|| u^{(j)}||_{1,2;\om}\]
for some constant $c$ depending only on $\om$. Since $u^{(j)}$ is by hypothesis weakly convergent in $W^{1,2}(\om,\R^2)$, the right-hand side is bounded uniformly in $j$, and hence, for some function $z(x,x_0)$, 
\[ t(x,x_0,u^{(j)}) \weakstar z(x,x_0)\]
in the sense of measures on $\hat{\om}$.   To show that $z(x,x_0)$ coincides almost everywhere with $t(x,x_{0},u)$ we let $f(x,x_0) \in C_c(\hat{\om})$ be arbitrary, define $U^{(j)}(x):=u^{(j)}(x)-u(x)$, and write
\begin{align*}\int_{\hat{\om}} t(x,x_0,u^{(j)}) f(x,x_0)\,dx\,dx_0 = & \int_{\hat{\om}} \adj \nabla U^{(j)}(x)  (U^{(j)}(x)-U^{(j)}(x_0))\cdot \frac{x-x_0}{|x-x_{0}|} f(x,x_0)\,dx \,dx_0 \\ &
+ \int_{\hat{\om}}\adj \nabla U^{(j)}(x) (u(x)-u(x_0)) \cdot \frac{x-x_0}{|x-x_{0}|} \,f(x,x_0)\,dx \,dx_0 \\ & + \int_{\hat{\om}} \adj \nabla u(x) (u^{(j)}(x)-u^{(j)}(x_0)) \cdot \frac{x-x_0}{|x-x_{0}|} \,f(x,x_0)\,dx\,dx_0. 
\end{align*}
The first integral converges to $0$ because $(x,x_0) \mapsto U^{(j)}(x)-U^{(j)}(x_0)$ converges strongly to $0$ in $L^{2}(\hat{\om})$; the second also converges to $0$, this time using the fact that $\nabla u^{(j)}(x)$ converges weakly to $0$ in $L^{2}(\om)$ together with the bounded convergence theorem.   The third integral tends to 
\[  \int_{\hat{\om}} \adj \nabla u(x) (u(x)-u(x_0)) \cdot \frac{x-x_0}{|x-x_{0}|} \,f(x,x_0)\,dx\,dx_0  \]
thanks to the strong convergence $u^{(j)} \to u$ in $L^{2}(\om)$.  It follows that $z(x,x_0)=t(x,x_0,u)$ a.e. in $\hat{\om}$.   We now apply \cite[Proposition A.3]{BM84} to the functional 
\[F(u,\om',r')=\int_{\hat{\om}} g(t(x,x_0,u))\,dx \,dx_0\]
and conclude that
\[\liminf_{j \to \infty} F(u^{(j)},\om',r') \geq F(u,\om',r').\] \qed
\end{proof}

\noindent
\textbf{Acknowledgement}\\[2mm]
The author would like to thank John Ball for his helpful comments on a preliminary version of this article.

\end{document}